\documentclass[11pt,reqno,a4paper]{amsart} 
\usepackage{amssymb,amsmath,bbm,bm,mathrsfs,epsfig,tikz}

\newtheorem{theorem}{Theorem}[section] 
\newtheorem{lemma}[theorem]{Lemma}
\newtheorem{proposition}[theorem]{Proposition}

\newtheorem{problem}[theorem]{Problem}

\numberwithin{equation}{section}

\def\1{{\mathbf 1}}
\def\build#1_#2^#3{\mathrel{\mathop{\kern 0pt#1}\limits_{#2}^{#3}}}

\def\epsilon{{\varepsilon}}
\def\phi{{\varphi}}
\def\C{{\mathbb C}}
\def\R{{\mathbb R}}
\def\Z{{\mathbb Z}}
\def\ZN{{\Z^{N}_{\, \downarrow}}}
\def\sgn{{\rm sgn}}

\def\tr{{\rm tr}}
\def\Tr{{\rm Tr}}
\def\U{{\rm U}}
\def\UC{{\mathbb U}}
\def\inv{{\rm inv}}

\def\L{\mathcal L}
\def\E{{\mathbb E}}

\def\diag{{\rm diag}}
\def\H{{\mathcal H}}
\def\I{{\mathcal I}}
\def\J{{\mathcal J}}

\def\Prob{{\mathcal M}}

\def\Supp{{\rm{Supp}}}
\def\PVint{{\rm PV}\!\! \int}

\def\MS{{\sf MS}}

\def\Up#1{{U^{#1}}}
\def\Gp#1{{G^{#1}}}
\def\Gpp#1{{G^{#1}_{+}}}
\def\Gpm#1{{G^{#1}_{-}}}
\def\Gppm#1{{G^{#1}_{\pm}}}
\def\Leb{{\rm Leb}}

\def\epsilon{{\varepsilon}}
\def\phi{{\varphi}}

\title[On The Douglas-Kazakov phase transition]{On The Douglas-Kazakov phase transition\\[4pt]
{\scriptsize Weighted potential theory under constraint for probabilists 
}} 
\author{Thierry L\'evy} \thanks{\noindent UPMC, Universit\'e Pierre et Marie Curie - LPMA - Case courrier 188 - 4, place Jussieu - 75252 Paris Cedex 05. Email : thierry.levy@upmc.fr }
\author{Myl\`ene Ma\"\i da} \thanks{Universit\'e Lille 1 - Laboratoire Paul Painlev\'e - Cit\'e Scientifique - 59655 Villeneuve d'ascq Cedex \\ Email : mylene.maida@math.univ-lille1.fr\\
This work was supported in part by the Labex CEMPI  (ANR-11-LABX-0007-01)}

\begin{document}
\maketitle

\begin{abstract} We give a rigorous proof of the fact that a phase transition discovered by Douglas and Kazakov in 1993 in the context of two-dimensional gauge theories occurs. This phase transition can be formulated in terms of the Brownian bridge on the unitary group $\U(N)$ when $N$ tends to infinity. We explain how it can be understood by considering the asymptotic behaviour of the eigenvalues of the unitary Brownian bridge, and how it can be technically approached by means of Fourier analysis on the unitary group. Moreover, we advertise some more or less classical methods for solving certain minimisation problems which play a fundamental role in the study of the phase transition. 
\end{abstract}

{\small \tableofcontents}

\section{Introduction} In 1993, Douglas and Kazakov discovered that the Euclidean $\U(N)$ Yang-Mills theory on the two-dimensional sphere exhibits a third order phase transition in the limit where $N$ tends to infinity (see \cite{DouKaz93}). More precisely, they discovered that the free energy of this model is not a smooth function of the total area of the sphere. Even more specifically, they found that the third derivative of the free energy has a jump discontinuity when the total area of the sphere crosses the value $\pi^{2}$.

It turns out that the partition function of the $\U(N)$ Yang-Mills theory on the two-dimensional sphere is exactly the same as that of the Brownian bridge on $\U(N)$, in a sense which we will explain below. In this correspondence, the area of the sphere becomes the lifetime of the bridge. The result of Douglas and Kazakov is thus equivalent to the existence of a phase transition concerning the limit as $N$ tends to infinity of the Brownian bridge on the unitary group $\U(N)$. In this limit, a certain quantity has a jump discontinuity when the lifetime of the bridge crosses the critical value $\pi^{2}$. This is the point of view which we adopt in this paper.

Let us give a precise statement of the main result. For all integer $N\geq 1$, let $\U(N)$ denote the group of unitary $N\times N$ matrices:
\[\U(N)=\{U\in M_{N}(\C) : UU^{*}=I_{N}\}.\]
Let us also denote by $\H_{N}$ the space of Hermitian $N\times N$ matrices:
\begin{equation}\label{HN}
\H_{N}=\{X \in M_{N}(\C) : X=X^{*}\}.
\end{equation}
Let us endow the real vector space $\H_{N}$ with the scalar product
\begin{equation}\label{prodscal}
\langle X,Y\rangle =N\Tr(XY).
\end{equation}
Let $(X_{t})_{t\geq 0}$ be the corresponding linear Brownian motion on $\H_{N}$. Let $(U_{t})_{t\geq 0}$ be the solution of the linear stochastic differential equation
\[dU_{t}=iU_{t}dX_{t}-\frac{1}{2}U_{t}\; dt\]
 in $M_{N}(\C)$, with initial condition $U_{0}=I_{N}$. The process $(U_{t})_{t\geq 0}$ is the Brownian motion on the unitary group $\U(N)$, issued from the identity. 

For all $t>0$, the distribution of $U_{t}$ admits a smooth positive density with respect to the Haar probability measure on $\U(N)$. We denote this density by $p_{N,t}$. The main result is the following.

\begin{theorem} \label{PT3} For all $T\geq 0$, the limit
\begin{equation}\label{defF}
F(T)=\lim_{N\to \infty} \frac{1}{N^{2}} \log p_{N,T}(I_{N})
\end{equation}
exists. The function $F:\R^{*}_{+}\to \R$ is $C^{2}$ on $\R^{*}_{+}$ and $C^{\infty}$ on $\R^{*}_{+}\setminus\{\pi^{2}\}$. Moreover, the third derivative of $F$ admits a jump discontinuity at the point $\pi^{2}$. More specifically, one has 
\[\lim_{\substack{T\to \pi^{2}\\ T<\pi^{2}}}F^{(3)}(T)=-\frac{1}{\pi^{6}} \ \mbox{ and } \lim_{\substack{T\to \pi^{2}\\ T>\pi^{2}}}F^{(3)}(T)=-\frac{3}{\pi^{6}}.\]
\end{theorem}

\begin{figure}[h!]
\begin{center}
\includegraphics[width=6.5cm]{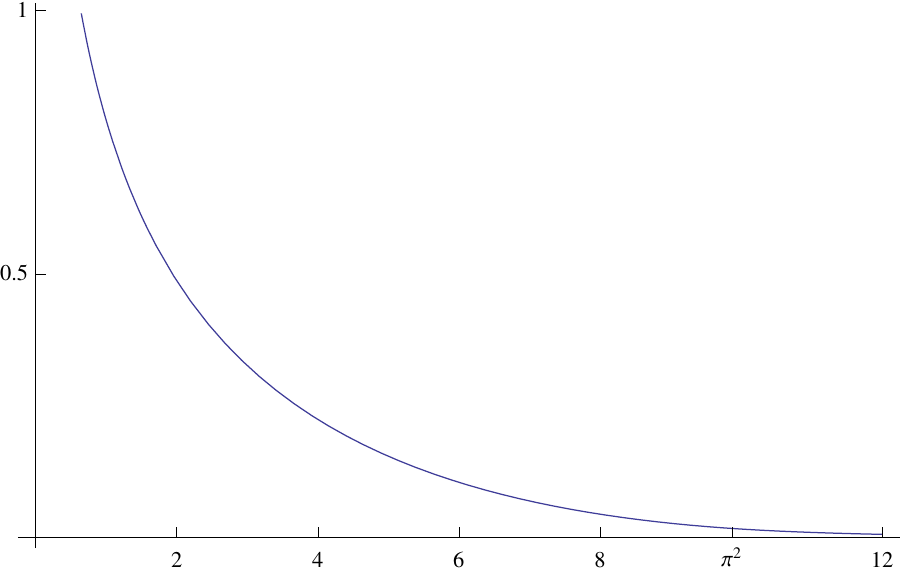} \hspace{0.3cm}  \raisebox{4mm}{\includegraphics[width=6.5cm]{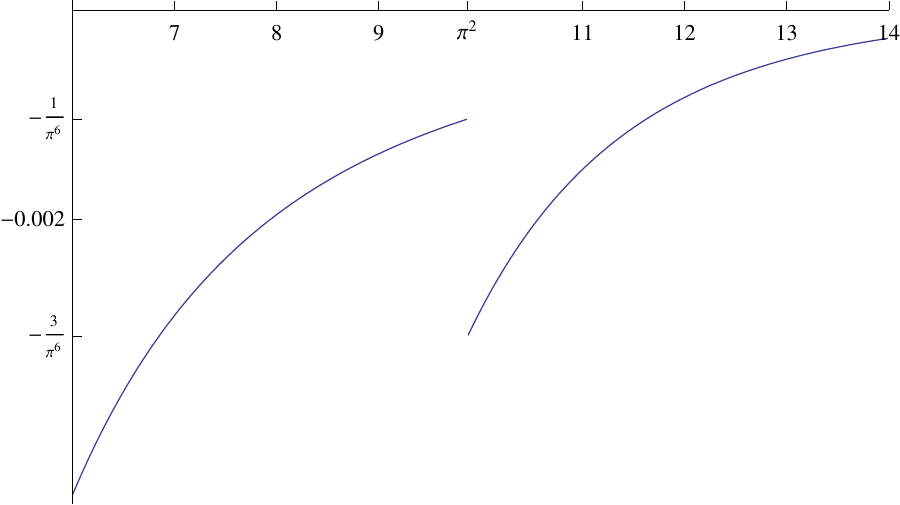}}
\caption{The function $T\mapsto F(T)$ (left) and its third derivative near the critical point $\pi^{2}$ (right).}
\end{center}
\end{figure}

In this paper, we give a complete proof of this theorem, based on the expression of the partition function as a series in Fourier space, that is, as a sum over Young diagrams. There is an alternative approach to this result, which focuses on the behaviour of the eigenvalues of the Brownian bridge (which we did not yet define). This point of view was taken in a previous work by Liechty and Wang \cite{LiechtyWang} and it has the advantage of making it plausible, at a heuristic level, that there is a phase transition. We will discuss it rather informally, to the extent that it allows one to better apprehend the nature of the phase transition. The work of Liechty and Wang, on the other hand, is technically arduous, and logically independent from ours. 

We also want to mention that several other third order phase transitions have been discovered over the last twenty years by physicists. In some recent works, Schehr and Majumdar draw a general picture and establish several general properties of these third order transitions. We refer the interested reader in particular to \cite{SchMaj13} and the many references therein.

Besides the desire to understand Theorem \ref{PT3}, one of our main motivations in writing this paper was to explore the concrete resolution of the minimisation problem which is at the heart of the study of the phase transition. In this problem, a probability measure on the real line is sought which minimises the sum of its energy of logarithmic self-interaction and its potential energy in an external field. This probability measure is moreover subject to a constraint, in our case that its density with respect to the Lebesgue measure should exist and not exceed $1$. 

The resolution of such minimisation problems seems to share some characteristics with the resolution of ordinary differential equations, where more or less undisclosable manipulations are performed in the secrecy of one's draft in order to identify the solution, which is then abruptly exhibited and proved to be the right one, using perfectly rigorous arguments which unfortunately give no hint as to the way in which the solution was found in the first place. 

In the present paper, we decided to disclose some of these shameful manipulations, with the hope that they can be fruitfully applied to other similar problems. The method which we explain is not at all new, we learned it in the treatise of Gakhov \cite{Gakhov}, and it is probably well known by many people who are better experts of these problems than we are. But Gakhov's book is written in a rather old-fashioned language and we did not find many other accounts of this method: we are thus trying to write the account of it which we would have often needed ourselves since we are studying problems of this kind. {\sl Sections \ref{sec:findmin} and \ref{sec:findweightedmin} are devoted to these formal derivations and we used a specific typography to emphasize that they do not have the same mathematical status as the rest of the paper.}

We will also try to illustrate the power of a number of tools and methods which we learned in the books and papers of Saff, Totik, Dragnev, Mhaskar, Kuijlaars, McLaughlin and others \cite{SafTot97, DraSaf97, MhaSaf92, KuijlaarsMcLaughlin, DraSaf00} in order to solve these minimisation problems. In particular, we have not shied away from quoting several results which we found particularly useful, and likely to be of  interest for the potential reader. To be brief, we have treated Theorem \ref{PT3} not only as an interesting result in itself, but also a case study in the art of solving minimisation problems under constraints.

Particle systems of the kind that will be studied in this paper are commonly referred to as ``Coulomb gases'' in the physical and mathematical literature. These systems have recently received much attention, as attested for example by the works of Sandier, Serfaty and their coauthors (see \cite{Ser} for a review), and also, closer to our point view, by the work of Chafa\"\i, Gozlan and Zitt in \cite{ChaGozZit14}, based on large deviations. In these works, very general results about the existence of the minimisers and their properties are stated, but only few minimisation problems are explicitly solvable, and even less when a constraint is added. In contrast, in this paper, we explore very concretely and as thoroughly as possible one very concrete example. \\

{\bf Acknowledgements.} We would like to acknowledge the stimulating effect which the recent work of Liechty and Wang \cite{LiechtyWang} had on us. In this work, we learned among many other things that it is actually possible to solve rigorously the minimisation problem which is the main object of this paper. We also would like to thank Bernhard Beckermann for his decisive help in the proof of Lemma \ref{beckermann}.

\section{Gaussian Wigner matrices}\label{Wigner}

It will be useful for us, as a warm-up, for the sake of coherence, and in order to set up some notation, to review the  case of Gaussian Wigner matrices in the light of the kind of minimisation problem which will appear
in the study of the phase transition. For this now very classical problem, we refer mainly to the monograph\cite{AndGuiZei10}.

\subsection{The Gaussian integral}\label{sec:GI}
For all real $t>0$, the Euclidean space $(\H_{N},\frac{1}{t} \langle \cdot ,\cdot\rangle)$ defined thanks to \eqref{HN} and \eqref{prodscal} carries a Gaussian probability measure which we shall denote by 
$\gamma_{N,t}$:
\[d\gamma_{N,t}(X)\propto e^{-\frac{N}{2t}\Tr(X^{2})} \; dX.\]
Consider a function $f:\H_{N}\to \R$ which is continuous, bounded and invariant by unitary conjugation in the sense that for all $X\in \H_{N}$ and all $U\in \U(N)$, one has $f(UXU^{-1})=f(X)$. 
Weyl's integration formula in this context (see section 2.5.4 in \cite{AndGuiZei10}) asserts that 
\[\int_{\H_{N}} f d\gamma_{N,t}=\frac{\left(\frac{N}{t}\right)^{\frac{N^{2}}{2}}}{1! \ldots N!} \int_{\R^{N}} f(\diag(\lambda_{1},\ldots,\lambda_{N})) V(\lambda_{1},\ldots,\lambda_{N})^{2} e^{-\frac{N}{2t}(\lambda_{1}^{2}+\ldots+\lambda_{N}^{2})} \; \frac{d\lambda_{1} \ldots d\lambda_{N}}{\sqrt{2\pi}^{N}},\]
where $V(\lambda_{1},\ldots,\lambda_{N})=\prod_{1\leq i<j\leq N} (\lambda_{i}-\lambda_{j})$ is the Vandermonde determinant. 
The constant in front of the integral can be determined by taking $f$ identically equal to $1$ and by writing $V(\lambda_{1},\ldots,\lambda_{N})$ as a determinant of monic Hermite polynomials. 

Introducing the empirical spectral measure
\[\hat \mu_{\lambda}=\frac{1}{N}\sum_{i=1}^{N} \delta_{\lambda_{i}},\]
the last integral can be rewritten as
\begin{equation}\label{integraleWigner}
\frac{\left(\frac{N}{t}\right)^{\frac{N^{2}}{2}}}{1! \ldots N!} \int_{\R^{N}} f(\diag(\lambda_{1},\ldots,\lambda_{N})) e^{{-N^{2} \J^{t}(\hat\mu_{\lambda})}}\; \frac{d\lambda_1\ldots d\lambda_N}{\sqrt{2\pi}^{N}},
\end{equation}
provided we define, for every Borel probability measure $\mu$ on $\R$ with compact support, and for every real $t>0$, 
\begin{equation}\label{defJ}
\J^{t}(\mu)=\iint_{\{(x,y)\in \R^{2} : x\neq y\}} -\log |x-y| \; d\mu(x)d\mu(y) + \frac{1}{2t} \int_{\R} x^{2} d\mu(x).
\end{equation}
One of the keys to the intuitive interpretation of $\J^{t}$ and the numerous similar functionals which we will encounter in this paper is the fact that the function $(z,w)\mapsto -\log |z-w|$ is the fundamental solution of the operator $-\frac{1}{2\pi}\Delta$ on the complex plane. Accordingly, the number $\J^{t}(\mu)$ can be understood as twice the electrostatic potential energy of a unit of electric charge distributed according to the measure $\mu$ 
in the complex plane, confined in this case to the real line, and subject to the external potential $x\mapsto \frac{x^{2}}{4t}$. One should however observe that $\J^{t}$ does not take into account the 
infinite energy of self-interaction of an atom of charge with itself.

\subsection{A large deviation principle} From \eqref{integraleWigner}, one naturally expects that, when $N$ tends to infinity, the eigenvalues of a typical Gaussian Hermitian matrix will be distributed 
so as to make $\J^{t}(\hat\mu_{\lambda})$ as small as possible. More precisely, since, as it will turn out, the functional $\J^{t}$ achieves its minimum at a unique diffuse probability measure, which we will denote by $\sigma_t,$ one expects that the probability for the empirical spectral measure $\hat \mu_\lambda$ to be far away from $\sigma_t$ will be exponentially small.
This was rigorously stated and proved by Ben Arous and Guionnet in \cite{BenArousGuionnet} in the framework of large deviations, as follows.

Let us denote by $\Prob(\R)$ the set of Borel probability measures on $\R$, endowed with the topology of weak convergence and with the corresponding Borel $\sigma$-field.

For each $N\geq 1$, let us denote by $\Lambda_{N,t}$ the distribution of the random empirical spectral measure $\hat\mu_{\lambda}$ of a Hermitian $N\times N$ matrix chosen under the probability measure $\gamma_{N,t}$. 

\begin{theorem}{\cite[Thm 1.3]{BenArousGuionnet}}\label{Alice} The sequence $(\Lambda_{N,t})_{N\geq 1}$ of probability measures on $\Prob(\R)$ satisfies a large deviation principle in the scale $N^{2}$ and with good rate function 
\[ \mu \mapsto \I^{t}(\mu)+\frac{1}{2}\log t-\frac{3}{4},\]
with 
\begin{equation}\label{defItWig}
\I^{t}(\mu)=\iint_{ \R^{2}} -\log |x-y| \; d\mu(x)d\mu(y) + \frac{1}{2t} \int_{\R} x^{2} d\mu(x).
\end{equation}
\end{theorem}

More details on the large deviation principle and on the properties of the functional $\I^{t}$ can of course be found in  the original paper \cite{BenArousGuionnet}, as well as in the book \cite{AndGuiZei10}. 

\subsection{A problem of minimisation} Let us make a couple of general remarks on the functional $\I^{t}$, most of which apply mutatis mutandis to all the functionals of the same type which we will consider in this paper. 

To start with, there is an issue of definition, because the first term of the right-hand side of \eqref{defItWig} is not defined for every $\mu\in\Prob(\R)$. This problem is solved by setting, for all $(x,y) \in \R^2$, 
$$ g_t(x,y)= - \log|x-y| +\frac{x^2+y^2}{4t}.$$
The function $g_{t}$ is bounded below and, for all $\mu\in\Prob(\R)$,
\[\I^{t}(\mu)=\iint_{ \R^{2}} g_t(x,y) \; d\mu(x)d\mu(y)\]
is well defined as an element of $(-\infty,\infty]$. Moreover, the map $\I^{t}$ itself is bounded below, and not identically equal to $\infty$.  Since the function $g_t$ is lower semi-continuous, the map $\I^{t}$ itself is lower semi-continuous in the weak topology. A minimising sequence for $\I^{t}$ has bounded second moment, so that it is tight, and $\I^{t}$ attains its infimum. It is also true, although less easy to prove, that $\I^{t}$ is a strictly convex function on $\Prob(\R)$ (see \cite[Lemma I.1.8]{SafTot97} or \cite[Lemma 2.6.2]{AndGuiZei10}), so that it attains its infimum at a unique probability measure. We shall denote this unique minimiser by $\sigma_{t}$ and our goal in the first part of this paper is to rediscover the shape of $\sigma_{t}$.

Before we turn to the determination of $\sigma_{t}$, let us introduce some notation and collect a few general facts related to minimisation problems slightly more general than that of the minimisation of $\I^{t}$. 

Consider $\mu\in \Prob(\R)$ with compact support. We denote by $\Up\mu$ the logarithmic potential of $\mu$, which is the function defined on $\R$  by 
\[\Up{\mu}(x)= \int_{\R} -\log |x-y|  \; d\mu(y).\]
We shall also make use of the Stieltjes transform $\Gp{\mu}$  of $\mu$ given, for all $z \in \C\setminus \Supp(\mu)$, by 
\[\Gp{\mu}(z)=\int_{\R} \frac{d\mu(y)}{z-y}.\]
On $\R\setminus \Supp(\mu)$, one has $(\Up{\mu})'=-\Gp{\mu}$. On the support of $\mu$, and under the assumption that $\mu$ has a sufficiently regular density, this relation takes a slightly different form, which we now explain.

For any H\"older continuous function $\psi : \R \rightarrow \R$ with compact support, we shall denote respectively by $U^{\psi}$ and $G^{\psi}$ the logarithmic potential and Stieltjes transform of the measure $\psi(x)dx$. We also set 
\[\PVint  \frac{\psi(y)}{x-y} \; dy = \lim_{h\to 0}\int_{|y-x|>h}  \frac{\psi(y)}{x-y} \; dy,\]
the principal value of this singular integral. Then, for all $x \in \R$, 
\begin{equation}\label{derUPV}
(\Up{\psi})'(x)= - \PVint  \frac{\psi(y)}{x-y} \; dy.
\end{equation}

The following theorem describes the class of minimisation problems which we are going to encounter, states for each of them the existence and uniqueness of its solution, and gives a characterisation of this solution based on its logarithmic potential. Given a closed subset $\Sigma$ of $\R$, we denote by $\Prob(\Sigma)$ the set of Borel probability measure on $\Sigma$. The logarithmic energy of a compactly supported measure $\mu\in \Prob(\Sigma)$ is the number ${\mathcal E}(\mu) \in (-\infty,+\infty]$ defined by
\begin{equation}\label{defLI}
{\mathcal E}(\mu)=\iint -\log|x-y| \; d\mu(x)d\mu(y).
\end{equation}
We say that a compact subset $K$ of $\C$ has positive capacity if there exists a probability measure supported by $K$ and with finite logarithmic energy.

\begin{theorem}[{\cite[Thm I.1.3 and I.3.3]{SafTot97}}]\label{theo:minsanscontrainte}
Let $\Sigma$ be a closed subset of $\R$ of positive capacity. Let $Q : \Sigma \rightarrow \R$ be a lower semi-continuous function such that $\lim_{|x| \rightarrow +\infty}(Q(x) -\log|x|)= +\infty$. Define, for any $\mu \in \Prob(\Sigma)$,
\begin{equation}\label{defIQ}
\I_Q(\mu)=\iint_{ \R^{2}} -\log |x-y| \; d\mu(x)d\mu(y) +  2\int_{\R} Q(x) d\mu(x).
\end{equation}
The infimum of $\I_Q$ over the set $\Prob(\Sigma)$ is finite and it is reached at a unique probability measure $\mu^{*},$ which is  compactly supported. 

Moreover, assume that $\mu$ is a compactly supported probability measure on $\Sigma$ with finite logarithmic energy and such that there exists a constant $F_Q$ for which 
\begin{equation}
 \tag{EL}\label{EL}
\left\{\begin{array}{ll}
U^{\mu}+Q = F_{Q} & \mbox{on } \Supp(\mu), \\
U^{\mu}+Q \geq F_{Q} & \mbox{on } \Sigma.
\end{array}\right. 
\end{equation}

Then $\mu=\mu^{*}$.
\end{theorem}

The measure $\mu^{*}$ is called the weighted equilibrium measure on the set $\Sigma$ in the external potential $Q$. 

The last assertion of this theorem is a powerful touchstone which allows one to check that a given measure is indeed the weighted equilibrium measure in a certain potential. It can be understood, in electrostatic terms, by observing that $U^{\mu}+Q$ is the electrostatic potential jointly created by the distribution of charge $\mu$ and the external potential $Q$. The fact that it is constant inside the support of $\mu$ indicates that the charges are at equilibrium inside this support. The fact that it takes greater values outside this support indicates that the charges are confined within the support, and that taking any small amount of charge outside the support would increase the electrostatic energy of the system. Otherwise stated, the system \eqref{EL} is nothing but what physicists call the Euler-Lagrange equations associated to this problem.

This characterisation does however not indicate how to find an expression of the measure $\mu^{*}$. We will now come back to our original problem and explain how a method which we learned in the treatise of Gakhov \cite{Gakhov}, and which is also explained in Section IV.3 of \cite{SafTot97},  allows one to find a reasonable candidate for the weighted equilibrium measure on $\R$ in the potential $Q(x)=\frac{1}{4t}x^{2}$. Then, in Section \ref{sec:minproof}, we will apply Theorem \ref{theo:minsanscontrainte} and give a rigorous proof that this candidate is indeed the minimiser. 

As explained in the introduction, what we explain here is not at all new, and can be regarded as a preparation for the much less classical minimisation problem under constraint that we will be facing in the second part of the paper.

{\sl
\subsection{Derivation of the weighted equilibrium measure}\label{sec:findmin}
Our goal in this section is to identify $\sigma_{t}$, the unique minimiser of $\I^{t}.$ As explained above, we will not care too much about rigour: what follows is not for the faint of heart. 

Let $t>0$ be fixed. We start by making the ansatz that $\sigma_t$ admits a density, which we denote by $\phi_t$, with respect to the Lebesgue measure, and that this density is regular enough, say H\"older continuous.

Since $\sigma_{t}$ is a minimiser of $\I_t$ over $\Prob(\R)$, and assuming all the required smoothness, $\I_t$ does not vary at first order when we add to $\sigma_{t}$ a small signed Borel measure of null total mass. A formal computation turns this observation into the assertion that the function
\begin{equation}\label{x-}
x\mapsto 2\int_{\R} -\log |x-y| \phi_{t}(y) \; dy +\frac{1}{2t} x^{2}
\end{equation}
must be constant on the support of $\sigma_t,$ which is nothing but the first equation of the system \eqref{EL} in this case.
Since the quadratic external potential grows much faster than the logarithmic interaction potential, we expect $\sigma_{t}$ to have compact support and, by symmetry, a support symmetric with respect to the origin. We postulate\footnote{More details about the support will be given in Section \ref{sec:supportMS}.} that this support is a symmetric interval $[-a,a].$ Therefore, differentiating \eqref{x-}  using \eqref{derUPV} leads us to the singular integral equation
\begin{equation}\label{PVeqWigner}
\PVint \frac{\phi_t(y)}{x-y} \; dy=\frac{x}{2t} \ \mbox{ for all } x\in [-a,a],
\end{equation}
where the unknown is the function $\phi_{t}$.

To solve this equation, we are going to use the Plemelj-Sokhotskyi formula (see for example \cite{Gakhov}), which asserts that if $\psi$ is H\"older continuous, then for every $x\in \R$,
\begin{equation}\label{PS}
\Gppm{\psi}(x)=
\PVint \frac{\psi(y)}{x-y} \; dy \mp  i \pi \psi(x),
\end{equation}
where we use the following notation : given a function $H$ analytic on $\C\setminus \R$ and a point $x$ on the real axis, the numbers $H_{\pm}(x)=\lim_{\epsilon \to 0} H(x\pm i\epsilon)$ are the limits from above and from below of $H$ at $x$.

The formula \eqref{PS} can equivalently be written as follows: for every $x\in \R$, we have
\begin{align}
\label{Gpv}&\displaystyle  \Gpp{\psi}(x)+\Gpm{\psi}(x) =2\ \PVint \frac{\psi(y)}{x-y} \; dy, \\ 
\label{Gpsi}&\Gpp{\psi}(x)-\Gpm{\psi}(x)=-2i\pi \psi(x).
\end{align}

Our strategy for solving \eqref{PVeqWigner} can be summarised as follows.  
\begin{itemize}
\item Equation \eqref{Gpsi}  allows us to extract $\phi_{t}$ from the knowledge 
of its Stieltjes transform $\Gp{\phi_{t}}$.
\item Our main goal will be to find $\Gp{\phi_{t}}.$ In view of \eqref{PVeqWigner} and \eqref{Gpv}, we will look for a function $H_{t}$ analytic on $\C\setminus \R$ which satisfies 
\begin{equation} \label{eqHt}
(H_t)_+(x) +(H_t)_-(x) = \frac{x}{t}, \ x\in [-a,a],
\end{equation}
and which has near infinity the behaviour of the Stieltjes transform of a probability measure, that is, $H_t(z) = z^{-1} + o(z^{-2})$. This will be our candidate for $\Gp{\phi_{t}}$.
\item Unfortunately, Equation \eqref{eqHt} is not particularly easy to solve directly. As \eqref{Gpsi} shows, what is easier is to find a function analytic on $\C\setminus \R$ with prescribed jump across the real axis. In order to bring \eqref{eqHt} into an equation of the form \eqref{Gpsi}, we introduce an auxiliary function, as we now explain. 
\end{itemize}

Let us introduce the function 
\[R(z)=\sqrt{z^{2}-a^{2}},\]
analytic on $\C\setminus [-a,a]$, and where we choose the branch of the square root which is a positive real number on $(a,\infty)$. For every $x\in (-a,a)$, we have
\[R_{\pm}(x)=\pm i \sqrt{a^{2}-x^{2}},\]
which is to say that the limits of $R$ on the real axis from above and from below are opposite. This is exactly what is needed to turn the sum which appears in \eqref{Gpv} into the difference which appears in \eqref{Gpsi}.

More explicitly, a function $H_{t}$ satisfies \eqref{eqHt} if and only if
$$ \left(\frac{H_t}{R}\right)_+(x) -  \left(\frac{H_t}{R}\right)_-(x) = -2i\pi \frac{x}{2\pi t\sqrt{a^2-x^2}}, \ x\in [-a,a].
$$
According to \eqref{Gpsi}, a solution of this equation is given by
\[H_t(z)=\frac{\sqrt{z^{2}-a^{2}}}{2\pi t}\int_{-a}^{a} \frac{y}{(z-y)\sqrt{a^{2}-y^{2}}} \; dy.\]
One computes
\[\int_{-a}^{a} \frac{y}{(z-y)\sqrt{a^{2}-y^{2}}} \; dy=\frac{z\pi}{\sqrt{z^{2}-a^{2}}}-\pi,\]
and finds
\[H_t(z)=\frac{1}{2t}(z-\sqrt{z^{2}-a^{2}}).\]
The parameter $a$ has still to be fixed. For this, we consider the behaviour of $H_{t}$ near infinity: we have $H_t(z)=\frac{a^{2}}{4t}z^{-1}+O(z^{-2})$. On the other hand, for $H_t$ to  be the Stieltjes transform of a probability measure, it should be equivalent to $z^{-1}$ at infinity. This entails
\[a=2\sqrt{t}.\]
Then, for each $x\in [-2\sqrt{t},2\sqrt{t}]$, we find the function $\phi_t$, our candidate to be a solution of \eqref{PVeqWigner}, by writing
\[\phi_t(x)=-\frac{1}{2i\pi} ((H_t)_+-(H_t)_-)(x)=\frac{1}{2\pi t}\sqrt{4t-x^{2}}.\]
This is the much expected semi-circular distribution. 
}

\subsection{Verification of the expression of the equilibrium measure}\label{sec:minproof} 
Let us now give a proof of the following proposition.

\begin{proposition}\label{proofmin}
 The probability measure $\sigma_{t}$ given by
\begin{equation}\label{sigmat}
d\sigma_{t}(x)=\frac{1}{2\pi t}\sqrt{4t-x^{2}}\; \1_{[-2\sqrt{t},2\sqrt{t}]}(x)\; dx.
\end{equation}
is the unique minimiser over $\Prob(\R)$ of the functional $\I^{t}$ defined by \eqref{defItWig}.
\end{proposition}

\begin{proof} We want to apply Theorem \ref{theo:minsanscontrainte}.
For this, we need to know the logarithmic potential of $\sigma_{t}$. Standard computations (see for example Theorem 5.1 in \cite{SafTot97}) show that 
\begin{align} \label{lepotWig}
\Up{\sigma_{t}}(x) &=\left\{\begin{array}{ll} -\frac{x^{2}}{4t}+\frac{1-\log t}{2} & \mbox{if } |x|\leq 2\sqrt{t}, \\[3pt]
\frac{-x^{2}+|x|\sqrt{x^{2}-4t}}{4t} - \log(|x|+\sqrt{x^{2}-4t})+\frac{1}{2}+\log 2 & \mbox{if } |x|> 2\sqrt{t}. \end{array}
\right.
\end{align}
Therefore, $U^{\sigma_{t}}(x)+\frac{x^{2}}{4t}$ is constant on the support of $\sigma_{t}$, equal to $\frac{1-\log t}{2}$, and an elementary verification shows that it is larger than this constant outside the support of $\sigma_{t}$. This suffices to guarantee that $\sigma_{t}$ is the minimiser of $\I^{t}$.
\end{proof}

\begin{figure}[h!]
\begin{center}
\includegraphics[width=5cm]{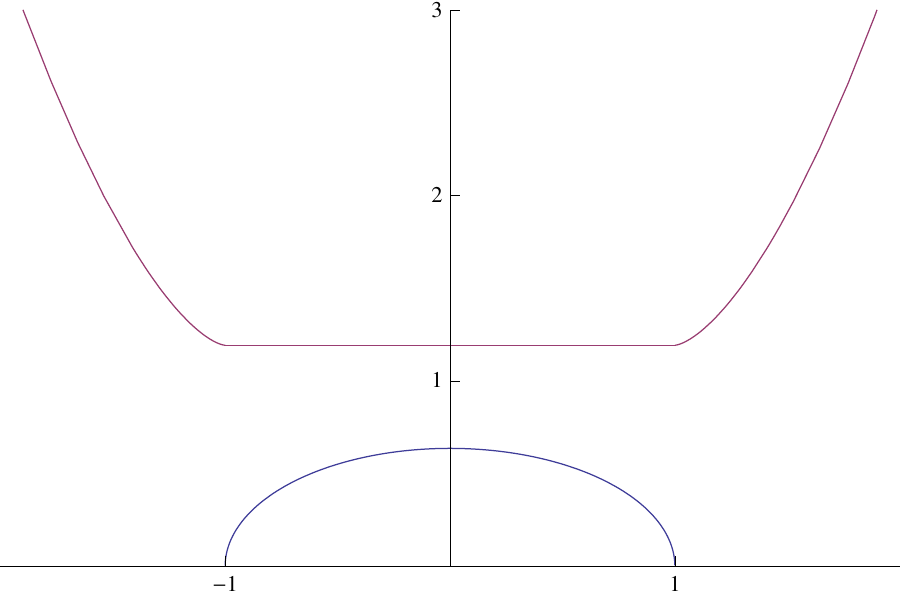}
\caption{The density of $\sigma_{t}$ and the graph of $x\mapsto \Up{\sigma_{t}}(x)+\frac{x^{2}}{4t}$ for $t=\frac{1}{4}$.}
\end{center}
\end{figure}

Since we know where $\I^{t}$ achieves its minimum, we can compute this minimum, and find
\[\I^{t}(\sigma_{t})=\int_{\R} \left(\Up{\sigma_{t}}(x)+\frac{x^{2}}{2t}\right) \; d\sigma_{t}(x)=\frac{1-\log t}{2}+\frac{1}{4t}\int_{\R} x^{2} \; d\sigma_{t}(x)=\frac{3}{4}-\frac{1}{2}\log t.\]
In particular, the rate function appearing in Theorem \ref{Alice} is equal to zero at $\sigma_t$, as expected.

\subsection{The support of the weighted equilibrium measure}\label{sec:supportMS}
Let us conclude this warm up by mentioning that we could have approached the determination of the minimiser from a slightly different viewpoint, which we will adopt and illustrate in great detail in the problem of minimisation under constraint which constitutes the main object of this paper. 

This alternative approach relies on a series of remarkable results of Mhaskar, Saff and Totik, which allows one to determine before anything else the support of the weighted equilibrium measure. Recall from \eqref{defLI} the definition of the logarithmic energy of a compactly supported probability measure on $\C$, and of the notion of positive capacity for a compact subset $K$ of $\C.$

It follows from a classical result of Frostman, of which Theorem \ref{theo:minsanscontrainte} is an elaborated version, that for every compact subset $K$ of $\C$ with positive capacity, there exists a unique probability measure $\omega^{K}$ supported by $K$ and with minimal logarithmic energy. This probability measure is called the equilibrium measure of $K$.

In the following theorem, and until the end of this section, we work under the assumptions and with the notations of Theorem \ref{theo:minsanscontrainte}.

\begin{theorem}[{\cite[Thm. IV.1.5]{SafTot97}}]\label{theo:MS}
For every compact subset $K$ of $\Sigma$ with positive capacity, set
\begin{equation}\label{eq:defMS}
\MS_{Q}(K)=\int (U^{\omega_{K}}+Q) \; d\omega_{K}.
\end{equation}
The support of the weighted equilibrium measure on $\Sigma$ in the external potential $Q$ is the smallest compact subset of $\Sigma$ with finite logarithmic energy which minimises the functional $\MS_{Q}$.
\end{theorem}

The functional $\MS_{Q}$ is called the Mhaskar and Saff functional, after the names of the authors who introduced it. Let us emphasize that the definition which we gave for this functional is the opposite of the original one, simply because we find it easier to consistently minimise functionals. 

Theorem \ref{theo:MS} alone does not suffice to identify the support of the weighted equilibrium measure, for there are too many compact subsets of $\Sigma$. However, the following theorem allows one, in many concrete situations, to restrict the search to a much narrower class of subsets.

\begin{theorem}[{\cite[Thm. IV.1.10 (b)]{SafTot97}}]\label{supportconvexe}
Consider an interval $I\subset \Sigma$ such that $Q$ is convex on $I$. Then the intersection of the support of $\mu^{*}$ with $I$ is an interval.
\end{theorem}

Using Theorems \ref{theo:MS} and \ref{supportconvexe} and an argument of symmetry, simple computations allow one to prove that the support of $\sigma_{t}$ is $[-2\sqrt{t},2\sqrt{t}]$. 

One of the advantages of knowing the support of the minimising measure beforehand is the possibility of applying the following result.

\begin{theorem}[{\cite[Thm. I.3.3]{SafTot97}}]\label{potquesurlesupport}
Consider a probability measure $\mu\in \Prob(\Sigma)$. Assume that $\nu$ has finite logarithmic energy, that $\Supp(\nu)\subset \Supp(\mu^{*})$ and that $U^{\nu}+Q$ is constant on $\Supp(\mu^{*})$. Then $\nu=\mu^{*}$.
\end{theorem}

Concretely, if we had known that the support of the weighted equilibrium measure on $\R$ in the external potential $x\mapsto \frac{x^{2}}{4t}$ was the interval $[-2\sqrt{t},2\sqrt{t}]$, then the first equality of \eqref{lepotWig} would have been sufficient to ensure that $\sigma_{t}$ was indeed the minimiser. In this situation, this may not seem to be a huge economy of effort, because the computation of $U^{\sigma_{t}}$ is relatively easy, but there are more complicated situations where the computation of the potential can be rather challenging.

\section{Brownian motion and Brownian bridge on the unitary group}

From now on, we focus on the main object of this paper, which is to investigate the Douglas-Kazakov phase transition. In this section, we introduce the Brownian bridge on the unitary group, which is one of the models on which this transition can be studied, and we explain how, at a heuristic level, one can understand the origin of the Douglas-Kazakov phase transition by considering the behaviour of the eigenvalues of this Brownian bridge.

\subsection{The partition function of the model}\label{sec:partitionfunction}
Recall from the introduction that the space $\H_{N}$ of $N\times N$ Hermitian matrices is endowed with the scalar product $\langle X,Y\rangle=N\Tr(XY)$. The linear Brownian motion in $\H_{N}$ is the 
Gaussian process $(X_{t})_{t \ge 0}$ with covariance specified by the following relation:
\[\forall A,B\in \H_{N}, \forall s,t\geq 0,\ \E[\langle X_{s},A\rangle \langle X_{t},B\rangle]=\min(s,t)\langle A,B\rangle.\]
The unique solution of the following linear stochastic differential equation in $M_{N}(\C)$:
\begin{equation}\label{edsBrownien}
dU_{t}=iU_{t}\; dX_{t} -\frac{1}{2}U_{t}\; dt, \textrm{ with initial condition } \ U_{0}=I_{N},
\end{equation}
is called the Brownian motion on the unitary group $\U(N)$ issued from the identity. By computing $d(U_{t}U_{t}^{*})$, one can check that almost surely, $U_t$ stays in $\U(N)$ for all $t \ge 0$.

The distribution of $U_t$ is absolutely continuous with respect to the Haar measure on $\U(N)$ and we will give in \eqref{Fourierp}  an expression of its density, which we denote by $p_{N,t}$.

Let  $L^{2}_{\inv}(\U(N))$ be the space of square-integrable functions with respect to the Haar measure on $\U(N)$ which are invariant by conjugation. This space admits an orthonormal basis indexed by the set $\ZN$ of decreasing sequences $\ell=(\ell_{1}>\ldots>\ell_{N})$ of $N$ integers, which is the Schur basis. For each $\ell\in \ZN$, the Schur function $s_{\ell}$ is the quotient of the restriction to $\U(N)$ of a polynomial function on $M_{N}(\C)$ and an integer power of the determinant. More precisely, if $U$ is a unitary matrix with pairwise distinct eigenvalues $z_{1},\ldots,z_{N}$, then
\[s_{\ell}(U)=\frac{\det(z_{i}^{\ell_{j}})_{i,j=1\ldots N}}{\det(z_{i}^{N-j})_{i,j=1\ldots N}}=\frac{\det(z_{i}^{\ell_{j}})_{i,j=1\ldots N}}{V(z_{1},\ldots,z_{N})},\]
with the notation $V$ for the Vandermonde determinant. The value of $s_{\ell}$ at matrices with multiple eigenvalues, for instance $I_{N}$, cannot be directly computed with this formula, but we have
\[s_{\ell}(I_{N})=\lim_{\alpha\to 0} s_{\ell}(\diag(1,e^{i\alpha},\ldots,e^{i(N-1)\alpha}))=\frac{V(\ell_{1},\ldots,\ell_{N})}{V(1,\ldots,N)},\]
an equality which is known as Weyl's dimension formula. We can use this formula to write the Fourier series of the Dirac mass at $I_{N}$:
\[\delta_{I_{N}} = \sum_{\ell\in \ZN} \frac{V(\ell_{1},\ldots,\ell_{N})}{V(1,\ldots,N)} s_{\ell}.\]
The equality above  is to be understood in the following distributional sense: for every smooth test function $f$ on $\U(N)$, one has $f(I_{N})=\sum_{\ell\in \ZN} \frac{V(\ell_{1},\ldots,\ell_{N})}{V(1,\ldots,N)} \int_{\U(N)} f(U) \overline{s_{\ell}(U)}\; dU$, the last integral being with respect to the normalised Haar measure.

For each $\ell\in \ZN$, the function $s_{\ell}$ is an eigenfunction of the Laplace operator $\Delta$ on $\U(N)$, which for our present purposes can conveniently be defined as twice the generator of the Brownian motion defined by \eqref{edsBrownien}. The corresponding eigenvalue is the non-positive real $-c_{2}(\ell)$, where
\begin{equation}\label{casimir}
c_{2}(\ell)=\frac{1}{N}\sum_{i=1}^{N}\left(\ell_{i}-\frac{N-1}{2}\right)^{2} - \frac{N^{2}-1}{12}.
\end{equation}
For each $t>0$, the Fourier series of the density of the heat kernel measure on $\U(N)$ at time $t$ (which is the distribution of $U_t$) can be obtained formally by applying the heat operator $e^{\frac{t}{2}\Delta}$ to the Fourier series of the Dirac mass at the identity. As shown for example in \cite[Thm 4.4 (a)]{Lia04}, the resulting series is the correct one and it is normally convergent. Thus, for all $U\in \U(N)$, we have
\begin{equation}\label{Fourierp}
p_{N,t}(U)= \sum_{\ell\in \ZN}e^{-\frac{t}{2}c_{2}(\ell)} \frac{V(\ell_{1},\ldots,\ell_{N})}{V(1,\ldots,N)} s_{\ell}(U).
\end{equation}


We can now define the Brownian bridge as follow. Choose a real $T>0$. The Brownian bridge of length $T$ on the unitary group is the unique pathwise continuous stochastic process $(B_{t})_{t\in [0,T]}$ on $\U(N)$ such that for all  $n\geq 1$, all reals $0=t_{0}< t_{1}< \ldots< t_{n}< t_{n+1}=T$ and all continuous function $f:\U(N)^{n}\to \R$, 
\[\E[f(B_{t_{1}},\ldots,B_{t_{n}})]=\frac{1}{Z_{N,T}}\int_{\U(N)^{n}}f(U_{1},\ldots,U_{n})\prod_{i=1}^{n+1} p_{N,t_{i}-t_{i-1}}(U_{i-1}^{-1}U_{i}) \; dU_{1}\ldots dU_{n},\]
with the convention $U_{0}=U_{n+1}=I_{N}$, and where we have set
\[Z_{N,T}=\int_{\U(N)^{n}}\prod_{i=1}^{n+1} p_{N,t_{i}-t_{i-1}}(U_{i-1}^{-1}U_{i}) \; dU_{1}\ldots dU_{n}.\]
This normalisation constant, or partition function, can be expressed at least in two ways: using the convolution property of the heat semigroup, it is equal to
\begin{equation}\label{defZp}
Z_{N,T}=p_{N,T}(I_{N}).
\end{equation}
Then, using \eqref{Fourierp}, we find that it can also be written as
\begin{equation}\label{defZl}
Z_{N,T}=\sum_{\ell\in \ZN}e^{-\frac{T}{2}c_{2}(\ell)} \frac{V(\ell_{1},\ldots,\ell_{N})^{2}}{V(1,\ldots,N)^{2}}.
\end{equation}

The number $Z_{N,T}$ is our main quantity of technical interest: the Douglas-Kazakov phase transition expresses a failure of smothness with respect to $T$ of its properly normalised limit as $N$ tends to infinity (see Theorem \ref{PT3}). Corresponding to the two descriptions of $Z_{N,T}$, there are two points of view on the phase transition, and two intuitive understandings of it.

The first point of view, which was that of Douglas and Kazakov in \cite{DouKaz93}, is based on the expression \eqref{defZl}, and is focused on the behaviour as $N$ tends to infinity of the positive measure on $\ZN$ of which $Z_{N,T}$ is the total mass. Since this is also the point of view which we adopt in the next sections, we will not dwell on it now. Let us simply summarise the main idea: when $N$ is large, the sum is dominated by a few terms which correspond to certain vectors $\ell$ which are close to an optimal vector, the ``shape'' of which depends on $T$ and undergoes a non-smooth change when $T$ crosses the critical value $\pi^{2}$. Incidentally, the reason why Douglas and Kazakov were interested in this quantity is that $Z_{N,T}$ is also the partition function of the 2-dimensional pure Euclidean Yang-Mills theory with structure group $\U(N)$ on a sphere of total area $T$. 

The other point of view is that of the ``collective field theory'', as Gross and Matytsin call it in \cite{GrossMatytsin}. It is more closely related to \eqref{defZp}, and is concerned with the behaviour of the eigenvalues of the Brownian bridge. Because this is not the point of view on which our technical approach is based, and because we find the phenomenon of transition in the behaviour of the eigenvalues particularly striking, we will now give a little more details about it in the two following short subsections.

\subsection{The eigenvalues of the Brownian motion}
Before discussing the eigenvalues of the Brownian bridge, it is appropriate to review the more classical and easier case of the Brownian motion. We want to understand the global behaviour of the spectrum of a unitary matrix picked under the distribution of $U_t$. 
We therefore denote by \[\hat \mu_{U}=\frac{1}{N}\sum_{i=1}^{N} \delta_{z_{i}}\] the empirical spectral measure of a unitary matrix $U$ with eigenvalues $z_{1},\ldots,z_{N}$. This is a probability measure on the unit circle $\UC$  of the complex plane.

To tackle the asymptotic behaviour of the law of $ \hat \mu_{U}$ under the heat kernel measure $p_{N,t}$, one could hope to mimic the strategy presented in Section \ref{Wigner} for the Hermitian case.
The function $p_{N,t}$ is invariant by conjugation on $\U(N)$, so that its value at a unitary matrix $U$ can, in principle, be expressed as a function of $\hat \mu_{U}$. Unfortunately, this is far from being as simple a function as the function $\mu\mapsto \exp(-N^{2} \J^{t}(\mu))$ which arose in the Hermitian case.

To be clear, an analysis of the distribution of the eigenvalues of the unitary Brownian motion analogous to the one which we reviewed in the Hermitian case is still out of reach at the time of writing.
A large deviation principle somewhat similar to Theorem \ref{Alice} is very likely to hold, but just to identify its rate function is still an open problem. Let us nevertheless mention the
partial results in this direction obtained, under the form of dynamical upper bounds, by Cabanal-Duvillard and Guionnet in \cite{CabDuvGui}.

While no large deviation principle is known, there exists a law of large numbers: the asymptotic distribution of the eigenvalues was computed by Biane in \cite{Biane} using moment techniques. To state his results, we use the notation $\tr$ for the normalised trace on $M_{N}(\C)$, the one such that $\tr(I_{N})=1$. Given a continuous function $f$ on $\UC$, we also use the notation $f(U)$ for the result of the application on $U$ of the functional calculus determined by $f$. Thus, $\tr f(U)=\int_{\U} f \; d\hat\mu_{U}$.

\begin{theorem}[\cite{Biane}] For all $t\geq 0$, there exists a probability measure $\nu_{t}$ on $\UC$ such that, for all continuous function $f:\UC\to \R$, one has the convergence in probability
\[\tr f(U_{t}) \build{\longrightarrow}_{N\to \infty}^{P} \int_{\UC} f \; d\nu_{t}.\]
The measure $\nu_{t}$ is determined by the fact that for all complex number $z$ close enough to $0$, 
\[\int_{\UC} \frac{1}{1-\frac{z}{z+1}e^{tz+\frac{t}{2}}\xi}\; d\nu_{t}(\xi)=1+z.\]
\end{theorem}

This description of the measure $\nu_{t}$ is rather indirect, but using it one can determine, among other things, the support of $\nu_{t}$. For each $t\in [0,4]$, the support of $\nu_{t}$ is the interval
\[\left\{e^{i\theta} : |\theta|\leq \frac{1}{2}\int_{0}^{t} \sqrt{\frac{4-s}{s}} \; ds\right\}.\]
For $t>4$, the support of $\nu_{t}$ is the full circle $\UC$. Moreover, for every $t>0$, the measure $\nu_{t}$ is absolutely continuous with respect to the uniform measure on $\UC$, and for every $t\neq 4$, this density is smooth and positive on the interior of the support of $\nu_{t}$. The density of $\nu_{4}$ has a singularity at $e^{i\pi}$ and is smooth and positive on $\UC\setminus\{e^{i\pi}\}$.

For the understanding of the Douglas-Kazakov phase transition, the most important fact is that the support of the limiting distribution $\nu_{t}$ is equal to $\{1\}$ when $t=0$, then grows symmetrically to become the full circle $\UC$ at $t=4$, and then stays equal to $\UC$. In other words, the choice which we made at the beginning for the scalar product on $\H_{N}$ (see the beginning of Section \ref{sec:partitionfunction}), and which determines the speed of the Brownian motion on $\U(N)$, is exactly such that we can see, at a macroscopic scale of time, the eigenvalues of $U_{t}$ progressively invade the whole unit circle.

It is interesting to note that, although several quantities change qualitatively when $t$ crosses the value $4$, for instance the decay of the moments of the measure $\nu_{t}$, 
which is polynomial for $t\leq 4$ and exponential for $t>4$, there does not seem to be a phase transition at $t=4$ in the same sense as there is for the Brownian bridge.

\subsection{The eigenvalues of the Brownian bridge} Let us now turn to the behaviour of the eigenvalues of the Brownian bridge, the study of which is distinctly harder than in the Brownian case.

It is possible to write down explicit expressions for the moments of the expected empirical spectral measure of the random matrix $B_{t}$, but to this day we have not been able to study their asymptotic behaviour directly. The analogue of Biane's result was obtained recently by Liechty and Wang in the formidable paper \cite{LiechtyWang}, using the point of view of determinantal processes. In particular, they prove the existence of a measure which plays for the Brownian bridge the role played for the Brownian motion by the measure $\nu_{t}$, and they give some  description of this measure. Note that this measure depends not only on the time $t$ at which the Brownian bridge is observed, but also on the lifetime $T$ of the bridge. Fortunately, we need not enter the details of Liechty and Wang's very technical solution to understand at least at a heuristic level why the phase transition predicted by Douglas and Kazakov occurs.

During the evolution of the Brownian bridge $(B_{t})_{t\in [0,T]}$, each eigenvalue starts from $1$ at time $0$, wanders in the unit circle and finally comes back to $1$ at time $T$. Collectively, the gas of eigenvalues undergoes an expansion on the unit circle, followed by a contraction. Since the processes $(B_{t})_{t\in [0,T]}$ and $(B_{T-t})_{t\in [0,T]}$ have the same distribution, we expect the evolution of this gas to be symmetric with respect to the time $\frac{T}{2}$, and its maximal expansion  to be reached precisely at $t=\frac{T}{2}$. 

Now suppose that $T$ is small. Then the time $\frac{T}{2}$ is too small for the eigenvalues to reach the other end of the unit circle, and the support of the limiting distribution of the eigenvalues does never expand enough to reach the point $e^{i\pi}$, which is the furthest from their starting point $1$.

If on the contrary $T$ is large, much larger than $4$, then the bridge
will see its eigenvalues fill the unit circle almost as quickly as those of the Brownian motion (although we would expect the influence of the conditioning to slow this expansion slightly). Hence, at $\frac{T}{2}$, and some time before and some time after, the support of the limiting distribution of the eigenvalues of $B_{t}$ fills the unit circle.

According to this informal discussion, we expect a transition between two regimes to occur for a critical lifetime $T$ of the bridge which is slightly larger than $8$. The discovery of Douglas and Kazakov, proved by Liechty and Wang, is that the transition occurs at $T=\pi^{2}$.

We will now give another proof of the existence of the phase transition, which is closer in spirit to the original point of view of Douglas and Kazakov, and which is logically independent of Liechty and Wang's work. Nevertheless, as mentioned in the introduction, Liechty and Wang's work gave us a crucial impulsion by showing that the minimisation which is at the core of Douglas and Kazakov's statement can indeed be solved rigorously.
 
\section{The shape of the dominant representation}\label{sec:representation}

In this section, we undertake the analysis of the partition function $Z_{N,T}$. We show that the computation of the free energy amounts to the resolution of a minimisation problem of the same kind as the one which we solved in the case of Wigner matrices, but with an additional constraint, of which we explain the nature and the origin. We then explain how to solve this problem, and we solve it. The computation of the free energy and the proof Theorem \ref{PT3} will be presented in Section \ref{sec:transition}.

\subsection{A discrete Gaussian integral}
Let us start from the expression \eqref{defZl} of the partition function. Using \eqref{casimir}, we find
\begin{align*}
Z_{N,T}&=\sum_{\ell\in \ZN}e^{-\frac{T}{2}c_{2}(\ell)} \frac{V(\ell_{1},\ldots,\ell_{N})^{2}}{V(1,\ldots,N)^{2}}\\
&=\frac{e^{\frac{T}{24}(N^{2}-1)}}{V(1,\ldots,N)^{2}} \sum_{\ell\in \ZN}e^{-\frac{T}{2N} \sum_{i=1}^N(\ell_i -\frac{N-1}{2})^2} V(\ell_{1},\ldots,\ell_{N})^{2},
\end{align*}
an expression very similar to the one which we obtained at the beginning of Section \ref{sec:GI}, and which indeed can be regarded as a discrete Gaussian integral. 

Since $V(\ell_{1},\ldots,\ell_{N})=V(\ell_{1}-\frac{N-1}{2},\ldots,\ell_{N}-\frac{N-1}{2})$, the last sum is easily expressed as a function of the empirical measure
\[\hat\mu_{\ell}=\frac{1}{N}\sum_{i=1}^{N} \delta_{\frac{1}{N}(\ell_i-\frac{N-1}{2})}.\]
We find
\begin{equation}\label{ZFourier}
Z_{N,T}=\frac{e^{\frac{T}{24}(N^{2}-1)+N(N-1)\log N}}{V(1,\ldots,N)^{2}} \sum_{\ell\in \ZN} e^{-N^{2} \J_T(\hat\mu_{\ell})},
\end{equation}
where we have set, for every Borel probability measure on $\R$,
\begin{equation}\label{defIstar}
\J_{T}(\mu)=\iint_{\{(x,y)\in \R^{2} : x\neq y\}} -\log |x-y| \; d\mu(x)d\mu(y) + \frac{T}{2} \int_{\R} x^{2} d\mu(x).
\end{equation}

Quite wonderfully, the functional $\J_{T}$ is none other than $\J^{\frac{1}{T}}$, as defined by \eqref{defJ}. Given this  coincidence, which of course is a reflection of a deeper correspondence between the eigenvalues of unitary matrices on one side and the integer vectors which index the Schur functions on the other side, the passage from $T$ to $\frac{1}{T}$ is the natural consequence of our working in Fourier space.

Recall that our main objective is to compute the free energy 
\[F(T)=\lim_{N\to\infty} \frac{1}{N^{2}}\log Z_{N,T}.\]
In the scale where we are working, the contribution of the prefactor of \eqref{ZFourier} is easy to compute. Indeed, a simple computation shows  that
\begin{equation}\label{prefacteur}
\frac{1}{N^{2}}\log \left(\frac{e^{\frac{T}{24}(N^{2}-1)+N(N-1)\log N}}{V(1,\ldots,N)^{2}}\right)=\frac{T}{24}+\frac{3}{2}+O\left(\frac{\log N}{N}\right).
\end{equation}
We will henceforward focus on the second term,
\begin{equation}\label{sommelog}
\frac{1}{N^{2}}\log \sum_{\ell\in \ZN} e^{-N^{2} \J_{T}(\hat\mu_{\ell})}.
\end{equation}

 

\subsection{Another large deviation principle}

Just as in our investigation of Wigner matrices, we expect the sum \eqref{sommelog} to be dominated by the terms corresponding to those $\ell$'s for which $\hat\mu_{\ell}$ minimises $\J_{T}$. Accordingly, we expect (correctly as we will see) this phenomenon of concentration to give rise to a large deviation principle, and we might also expect (incorrectly as it turns out) that the simple relation 
\begin{equation}\label{Finsature}
F(T)=\frac{T}{24}+\frac{3}{2}-\I_{Q_{T}}(\sigma_{1/T})=\frac{T}{24}+\frac{3}{4}-\frac{1}{2}\log T
\end{equation}
holds, where $Q_{T}(x)=\frac{T}{4}x^{2}$, the functional $\I_{Q_{T}}$ is defined by \eqref{defIQ}, and 
 $\sigma_{1/T}$ is the minimiser defined by the right-hand side of \eqref{sigmat}.

Let us start by stating the  large deviation principle. We will then explain why \eqref{Finsature} is not true, or at least not for all values of $T$.

As always, the set $\Prob(\R)$ of Borel probability measures on $\R$ is endowed with the weak topology and the corresponding Borel $\sigma$-field. Another set of probability measures on $\R$ appears in the statement, namely the set
\begin{equation}\label{defsetL}
\L(\R)=\{\mu \in \Prob(\R) : \forall a,b\in \R, a<b, \mu([a,b])\leq b-a\}
\end{equation}
of Borel probability measures on $\R$ which are absolutely continuous with respect to the Lebesgue measure, with a density not greater than $1$. 

The following result can be viewed as a ``discrete analogue'' of Theorem \ref{Alice}. 

\begin{theorem}[{\cite[Thm 2]{GuiMai05}}] Let $T>0.$ For each $N\geq 1$, let $\Pi_{N}$ be the Borel measure on $\Prob(\R)$ defined by
\[\Pi_{N}=\sum_{\ell\in \ZN} e^{-N^{2} \J_{T}(\hat\mu_{\ell})} \delta_{\hat\mu_{\ell}}.\]
The sequence of measures $(\Pi_{N})_{N\geq 1}$ satisfies on $\Prob(\R)$ a large deviation principle in the scale $N^{2}$ with good rate function $\widetilde \I_{Q_{T}}$ defined as follows:
\[\forall \mu \in \Prob(\R), \; \widetilde\I_{Q_{T}}(\mu)=\left\{\begin{array}{ll} \I_{Q_{T}}(\mu)& \mbox{if } \mu \in \L(\R),\\ +\infty & \mbox{otherwise}.
\end{array}\right.\]
\end{theorem}

\subsection{A problem of minimisation under constraint}

The fact that the rate function $\widetilde\I_{Q_{T}}$ of this large deviation principle is equal to $+\infty$ outside $\L(\R)$ indicates that the measures $\hat \mu_{\ell}$ do not explore the whole set $\Prob(\R)$. This is not surprising if we realise that the integer vectors $\ell$ are submitted to the constraint that their components must be pairwise distinct. Indeed, this constraint puts an upper bound to the concentration of mass which the measures $\hat \mu_{\ell}$ can achieve, and we have
\[\bigcap_{p\geq 1} \overline{\bigcup_{N\geq p} \{\hat\mu_{\ell} : \ell \in \ZN\}}=\L(\R).\]

Since $\L(\R)$ is a closed subset of $\Prob(\R)$, the same arguments which we used after the definition of the functional $\J^{t}$ (see \eqref{defJ}) imply that the functional $\I_{Q_{T}}$ attains its infimum on $\L(\R)$ at a unique element of $\L(\R)$, which we shall denote by $\mu^{*}_{T}$. A precise statement of this fact can be found in \cite[Thm 2.1]{DraSaf97}, and is stated below, for the convenience of the reader, as Theorem \ref{theo:minaveccontrainte}.

The existence and uniqueness of $\mu^{*}_{T}$, combined with the large deviation principle and with \eqref{prefacteur}, yields the following result, which is the key to our computation of the free energy.

\begin{proposition}\label{limZ}
For all $T>0$, one has
\[F(T)=\lim_{N \to \infty}  \frac{1}{N^2}\log  Z_{N,T} = \frac{T}{24} + \frac{3}{2} - \inf\{\I_{Q_{T}}(\mu): \mu \in \L(\R)\}=\frac{T}{24} + \frac{3}{2} - \I_{Q_{T}}(\mu^{*}_{T}).\]
\end{proposition}

Now, the reason why \eqref{Finsature} is not true for all $T$ is easy to explain. The problem is that $\sigma_{1/T}$, which is the minimiser of $\I_{Q_{T}}$ over the whole space $\Prob(\R)$, does not belong to $\L(\R)$ for all $T$. In fact, a very simple computation shows that $\sigma_{1/T}$ belongs to $\L(\R)$ exactly when $T\leq \pi^{2}$. In this case, $\mu^{*}_{T}=\sigma_{1/T}$ and the equality \eqref{Finsature} is true, giving the correct value of the free energy.

For $T>\pi^{2}$ on the other hand, the absolute minimiser of $\I_{Q_{T}}$ over $\Prob(\R)$ does not belong to $\L(\R)$. We are thus facing a minimisation problem of a new kind, which is the following.

\begin{problem} Fix $T>\pi^{2}$. Set $Q_{T}(x)=\frac{T}{4}x^{2}$. Find $\mu^{*}_{T}$, the unique minimiser  on the set $\L(\R)$ (defined by \eqref{defsetL}) of the functional $\I_{Q_{T}}$ (defined by \eqref{defIQ}).
\end{problem}

The resolution of this problem will occupy us almost until the end of this paper. Until the end of Section \ref{sec:representation}, we always assume that $T>\pi^{2}.$

\subsection{Strategy, tools, and outline of the results} 

Let us start by quoting the result which ensures the existence and uniqueness of the constrained minimiser $\mu^{*}_{T}$. This result is a version of Theorem \ref{theo:minsanscontrainte} adapted to problems of minimisation under constraint. It is proved in a paper of Dragnev and Saff \cite{DraSaf97}, where the authors give the first systematic study of minimisation problems under constraint. 

\begin{theorem}[{\cite[Thm 2.1]{DraSaf97}}]\label{theo:minaveccontrainte}
Let $\Sigma$ be a closed subset of $\R$ of positive capacity. Let $Q : \Sigma \rightarrow \R$ be a lower semi-continuous function such that $\lim_{|x| \rightarrow +\infty}(Q(x) -\log|x|)= +\infty$. Let $\sigma$ be a Borel measure on $\Sigma$ such that $\Supp(\sigma)=\Sigma$ and $\sigma(\Sigma)>1$. Assume that the restriction of $\sigma$ to every compact set has finite logarithmic energy. 

Then the infimum of the functional $\I_Q$ over the set $\{\mu \in \Prob(\Sigma) : \mu\leq \sigma\}$ is finite and it is reached at a unique probability measure $\mu^{*},$ which is  compactly supported. 

Moreover, consider a compactly supported probability measure $\mu$ on $\Sigma$ such that $\mu\leq \sigma$. Then $\mu=\mu^{*}$ if and only if there exists a constant $F_Q$ such that
\begin{equation}
\tag{EL$_c$}\label{ELc}
\left\{\begin{array}{ll}
U^{\mu}+Q \leq F_{Q} &  \mu \mbox{ almost everywhere (a.e.)} \\
U^{\mu}+Q \geq F_{Q} & \sigma-\mu \mbox{ a.e.}
\end{array}\right.
\end{equation}

\end{theorem}
The measure $\mu^{*}$ is called the constrained weighted equilibrium measure on $\Sigma$ in the external field $Q$ and with constraint $\sigma$.

Apart from a slight variation in the technical assumptions, the main difference between this result and Theorem \ref{theo:minsanscontrainte} lies in the characterisation of the minimising measure by means of its potential. This new characterisation can be paraphrased by saying that $U^{\mu}+Q$ is equal to $F_{Q}$ on the part of the support of $\mu$ where the constraint $\mu\leq \sigma$ is not saturated, smaller than $F_{Q}$ on the set where the constraint $\mu\leq \sigma$ is  saturated, and larger than $F_{Q}$ outside the support of $\mu$.

In comparison with the case without constraint, the new phenomenon is that on an interval where $\mu$ puts as much mass as it is allowed to, namely the mass given by $\sigma$, it would in general have been energetically more efficient to put even more mass. On such an interval, $U^{\mu}+Q$ is smaller than $F_{Q}$ and, should the constraint be released, some of the mass of $\mu$ initially located outside this interval would migrate into it (see Figure \ref{potsat} for an explicit example).

We shall often refer to the characterisation of the constrained weighted equilibrium measure in terms of its potential, given by the system \eqref{ELc}, as the Euler-Lagrange formulation of the minimisation problem.

Let us turn to our specific problem. Considering the parity of the external potential $Q_{T}$ and of the constraint, which is the Lebesgue measure, considering the relative position of the densities of the absolute minimiser $\sigma_{1/T}$ and the Lebesgue measure, and using the electrostatic intuitive formulation of the problem, in which the constraint can be understood as the specification of a finite electric capacitance of the set on which the measures live, it seems reasonable to expect that there will exist two reals $\alpha$ and $\beta$ with $0<\alpha<\beta$ such that $\mu^{*}_{T}$, the constrained minimiser, is equal to the Lebesgue measure on $[-\alpha,\alpha]$, is strictly between zero and the Lebesgue measure on $(-\beta,-\alpha)\cup (\alpha,\beta)$, and vanishes outside $(-\beta,\beta)$. This is illustrated on Figure \ref{sigmamu} below.

\begin{figure}[h!]
\begin{center}
\includegraphics{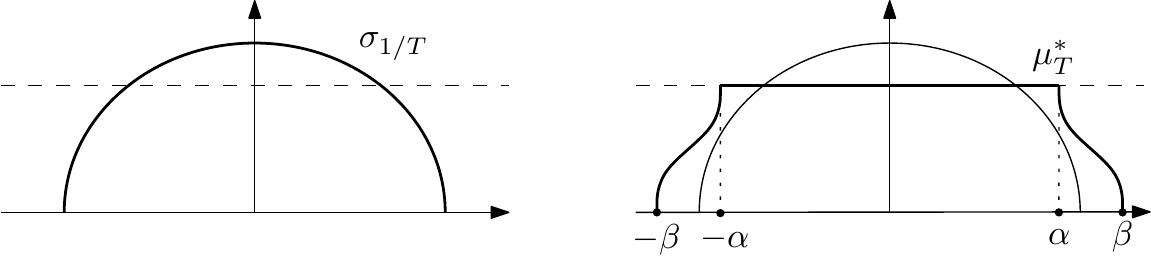}
\caption{\label{sigmamu} A schematic representation of the measures $\sigma_{1/T}$ and $\mu^{*}_{T}$.}
\end{center}
\end{figure}

According to the general principle, advertised for instance by Saff and Totik in their book \cite{SafTot97}, that the first step in the determination of a minimising measure is the determination of its support, our first task will be to formulate this guess precisely, to prove it, and to determine explicitly the values of $\alpha$ and $\beta$ in function of $T$. 

The first result which we will prove is the following. We use the notation $\Leb$ for the Lebesgue measure.

\begin{lemma}
 \label{beckermann}
There exists two reals $0< \alpha < \beta$ such that the measure $\mu^{*}_{T}$ satisfies
\[\Supp (\mu^{*}_{T})=[-\beta,\beta] \ \ \mbox{ and } \ \ \Supp ({\rm Leb}-\mu^{*}_{T})=\R\setminus (-\alpha,\alpha).\]
\end{lemma}

The proof of this lemma involves three ingredients. The first was suggested to us by Bernd Beckermann. It consists in exploiting the symmetry of the problem and transforming it into a problem on $\R^{+}$ via the squaring map. The second ingredient is a duality argument, which the form of the statement suggests a posteriori, and which consists in studying $\mu^{*}_{T}$ and $\Leb-\mu^{*}_{T}$ on the same footing. We learned this technique in a paper of Kuijlaars and Mc Laughlin (\cite{KuijlaarsMcLaughlin}, see also \cite[Corollary 2.10]{DraSaf97}). The third ingredient is the following version of Theorem \ref{supportconvexe} for minimisation problems under constraint, which yields convexity properties of the support.

\begin{theorem}[{\cite[Thm 2.16 (a)]{DraSaf97}}]\label{supportconvexecontrainte} Under the assumptions of Theorem \ref{theo:minaveccontrainte}, the intersection of the support of $\mu^{*}$ with any interval contained in $\Sigma$ and on which $Q$ is convex is an interval.
\end{theorem}

Let us emphasize one of the great strengths of this theorem, which is not to depend on any assumption on the constraining measure $\sigma$.

The actual determination of $\alpha$ and $\beta$ is based on the use of the Mhaskar-Saff functional (see Theorem \ref{theo:MS}). It requires a certain amount of computation of integrals, which turn out to be elliptic integrals. The formulation of the result requires the definition of the complete elliptic integrals of the first and second kind, which are the two functions of a parameter $k\in (0,1)$ defined by

\begin{equation}\label{defEK}
K(k)=\int_{0}^{1}\frac{ds}{\sqrt{(1-s^{2})(1-k^{2}s^{2})}} \ \ \mbox{ and } \ \ E(k)=\int_{0}^{1}\sqrt{\frac{1-k^{2}s^{2}}{1-s^{2}}} \; ds.
\end{equation}
There is a tradition, which we follow, not to write the dependence of $K(k)$ and $E(k)$ on $k$, and to write simply $K$ and $E$.

\begin{lemma}\label{MSab} The two reals $\alpha$ and $\beta$ of which Lemma \ref{beckermann} assures the existence are uniquely determined, as functions of $T$, by the relations
\[\beta=\frac{1}{2E-(1-k^{2})K} \mbox{ and } \alpha=k\beta,\]
where $E$ and $K$ are the complete elliptic integrals defined by \eqref{defEK}, and $k$ is the unique element of $(0,1)$ such that
\[T=8EK-4(1-k^{2})K^{2}.\]
\end{lemma}

We also refer the reader to Figure \ref{figdensite}.\\

Once $\alpha$ and $\beta$ are determined, there remains to compute the restriction of $\mu^{*}_{T}$ to the set $J(\alpha,\beta)=[-\beta,-\alpha]\cup [\alpha,\beta]$. For this, we use the following fact, which is a consequence of the result stated as Corollary 2.9 in \cite{DraSaf97}.

\begin{proposition}\label{levecontrainte} The restriction of $\mu^{*}_{T}$ to $J(\alpha,\beta)=[-\beta,-\alpha]\cup [\alpha,\beta]$ is $(1-2\alpha)$ times the weighted equilibrium mesure on the set $(-\infty,-\alpha]\cup [\alpha,+\infty)$ in the external potential $Q_{T,\alpha}=\frac{1}{1-2\alpha}(Q_{T}+U^{\Leb_{|[-\alpha,\alpha]}})$.
\end{proposition}

At this point, the problem has been transformed into a minimisation without constraint, on a new subset of $\R$, with a new potential. We will derive the solution of this problem using the same method which we used in Section \ref{sec:findmin}, be it to the price of more complicated computations, and we will check that the solution which we obtain is indeed the minimiser using Theorem \ref{potquesurlesupport}. 

The last step of our proof of the phase transition consists in an explicit computation of the free energy of the model, and in a direct computation of the limits on the left and on the right at $\pi^{2}$ of the third derivative of the free energy. This will be done in Section \ref{sec:transition}.

\subsection{The support of the weighted equilibrium potential}

In this section, we prove Lemma \ref{beckermann} and Lemma \ref{MSab}. 

\begin{proof}[Proof of Lemma \ref{beckermann}] 

In this proof, for the sake of simplicity, we shall denote the functional $\I_{Q_{T}}$ by $\I_{T}$.

{\em Step 0: the minimiser is even.} The fact that the potential $Q_{T}(z) = \frac{T}{4}z^2$ is even implies that the functional $\I_{T}$ takes the same value on a measure $\mu$ and on the measure $\mu^{\vee}$ which is the image of $\mu$ by the map $x\mapsto -x$. Moreover, since $\I_{T}$ is convex, this common value cannot be smaller than the value of $\I_{T}$ on the symmetric measure $\frac{1}{2}(\mu+\mu^{\vee})$. Hence, $\mu^{*}_{T}$, the minimiser of $\I_{T}$, must be a symmetric measure.

The first idea of this proof is to exploit this symmetry and to transform the problem of minimisation of $\I_{T}$ into an equivalent but simpler problem on $\R_{+}$, via the squaring map. The reason why this new problem is simpler is that $0$ acts as a ``hard'' boundary for the problem, and reduces the freedom with which measures are allowed to move. 

{\em Step 1: folding up the problem.} Let $s:\R\to \R_{+}$ denote the squaring map $x\mapsto s(x)=x^{2}$. Let us define $\nu^{*}_{T}=\mu^{*}_{T}\circ s^{-1}$, the image measure of $\mu^{*}_{T}$ by $s$. We claim that $\nu^{*}_{T}$ is the unique minimiser of the functional 
\[\I_T^{s}(\nu) =    \iint_{(\mathbb R_+)^2} -\log|u-v| \; d\nu(u) d\nu(v)+ T\int_{\mathbb R_+} u\, d\nu(u)\]
on 
\[\L^{s}=\left\{\nu \in \Prob(\R_{+}) : \forall a,b\in \R_{+}, a<b, \nu([a,b])\leq 2(\sqrt{b}-\sqrt{a}) = \int_{a}^{b}\frac{du}{\sqrt{u}}\right\}.\]
Indeed, a simple change of variables implies that for every $\mu\in \L$, the measure $\mu\circ s^{-1}$ belongs to $\L^{s}$ and, if $\mu$ is symmetric, that the equality
\[2\,\I_{T}(\mu)=\I_{T}^{s}(\mu\circ s^{-1})\]
holds. Now, every measure $\nu$ in $\L^{s}$ is the image measure by the squaring map of a unique symmetric measure $\mu$ in $\L$, and we have $\I^{s}_{T}(\nu)=2\,\I_{T}(\mu) \geq 2\, \I_{T}(\mu^{*}_{T}) =\I^{s}_{T}(\nu^{*}_{T})$.

It must be noted that the squaring map plays a special role here: the logarithmic potential behaves particularly nicely under this transformation, and the quadratic part of the functional $\I_{T}^{s}$ is still the logarithmic energy.

Let us denote by $\kappa$ the measure $du/\sqrt{u}$ on $\R_{+}$, which is the image of the Lebesgue measure by the squaring map. Proving the lemma is equivalent to proving that there exists two reals $a$ and $b$ with $0<a<b$ such that
\[\Supp(\nu^{*}_{T})=[0,b] \;\; \mbox{\textrm and} \;\; \Supp(\kappa-\nu^{*}_{T})=[a,+\infty).\]

{\em Step 2: restricting to a compact set.} We already said that our main reference for constrained energy problems is the paper of Dragnev and Saff \cite{DraSaf97}. In this paper, the authors work under the assumption that the constraint is a finite measure, an assumption which is not satisfied by the measure $\kappa$ on $\R_{+}$. Nevertheless, the proof of the existence of the minimiser (pages 243 to 246 in \cite{DraSaf97}) uses only the fact that the constraint is finite on compact subsets of $\C$, which is true for our measure $\kappa$. Moreover, as the authors explain in Remark 2.2, their proof of the existence of the minimiser shows that it has compact support. 

Thus, the arguments of \cite{DraSaf97} show that $\nu^{*}_{T}$ has compact support. Let us fix $M>0$, which may depend on $T$, such that the support of $\nu^{*}_{T}$ is a subset of $[0,M)$. Let us emphasize that, for a reason which shall become clear soon, we choose $M$ large enough that it does not belong to the support of $\nu^{*}_{T}$. Then $\nu^{*}_{T}$ is the unique minimiser of $\I^{s}_{T}$ on 
\[\L^{s}_{M}=\left\{\nu \in \Prob([0,M]) : \nu\leq \kappa_{|[0,M]}\right\}.\]
Since the measure $\kappa_{|[0,M]}$ is finite, we are now exactly in the framework of the paper of Dragnev and Saff, and we can apply their results. 

{\em Step 3: existence of $b$.}  Since the external potential $q_{T}(z) = \frac{T}{2}z$ to which $\nu^{*}_{T}$ is subjected is convex, Theorem \ref{supportconvexecontrainte}  implies that the support of $\nu^{*}_{T}$ is an interval. Since $q_T$ is increasing and the density of $\kappa$ is decreasing, this interval must contain $0$, for if it did not, a translation of $\nu^{*}_{T}$ to the left would produce another element of $\L^{s}$ achieving a strictly smaller value of $\I^{s}_{T}$. Altogether, this proves that the support of $\nu^{*}_{T}$ is of the form $[0,b]$ for some $b>0$, indeed $b\geq \frac{1}{4}$, since $\kappa([0,\frac{1}{4}])=1$.

{\em Step 4: existence of $a$.} We will now prove that the support of the difference $\kappa_{|[0,M]}-\nu^{*}_{T}$ is an interval of the form $[a,M]$, which is to say that the measure $\nu^{*}_{T}$ saturates the constraint $\kappa$ on the interval $[0,a)$. For this, we use the Euler-Lagrange formulation of the minimisation problem, that is, the characterisation of the minimising measure in terms of its potential (see Theorem \ref{theo:minaveccontrainte}), and a clever  duality argument which we borrow from Kuijlaars and Mc Laughlin (see \cite{KuijlaarsMcLaughlin}, and also \cite[Corollary 2.10]{DraSaf97}). A glance at Figure \ref{nunu} may help to memorise the definitions of the measures which we introduce.

\begin{figure}[h!]
\begin{center}
\scalebox{0.85}{\includegraphics{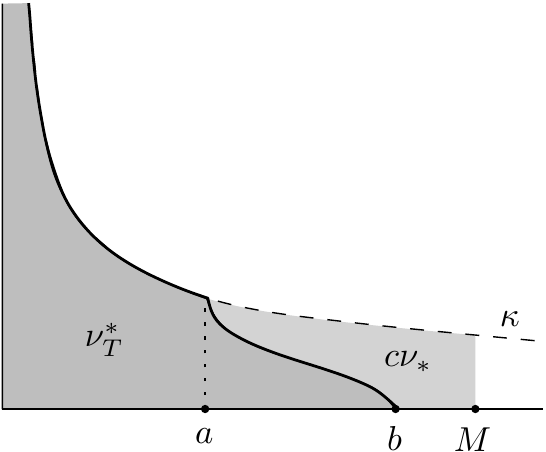}}
\caption{\label{nunu}The measures $\kappa$, $\kappa_{M}=\kappa_{|[0,M]}$, $\nu^{*}_{T}$ and $\nu_{*}=\frac{1}{c}(\kappa_{M}-\nu^{*}_{T})$, with $c=2\sqrt{M}-1$.}
\end{center}
\end{figure}

Let us denote $\kappa_{M}=\kappa_{|[0,M]}$. According to the Euler-Lagrange formulation of the minimisation problem on $[0.M]$ of which $\nu^{*}_{T}$ is the solution, there exists a constant $\ell$ such that 
\begin{equation}\tag{EL${}^{*}$}\label{ELnu}
\left\{\begin{array}{ll}
U^{\nu^{*}_{T}}+q_{T} \leq \ell & \nu^{*}_{T}  \ \textrm{a.e.} \\
U^{\nu^{*}_{T}}+q_{T} \geq \ell & \kappa_{M}-\nu^{*}_{T} \ \textrm{a.e.}
\end{array}
\right.
\end{equation}

Let us now set $c=2\sqrt{M}-1$ and define the probability measure $\nu_{*}=\frac{1}{c}(\kappa_{M}-\nu^{*}_{T})$. By merely rewriting the system \eqref{ELnu}, we see that $\nu_{*}$ is the unique probability measure on $[0,M]$ which is dominated by $\frac{1}{c}\kappa_{M}$ and  for which there exists a constant $\ell$, the same as above, such that
\begin{equation}\tag{EL${}_{*}$}\label{ELdual}
\left\{\begin{array}{ll}
U^{\nu_{*}}+\frac{1}{c}(-U^{\kappa_{M}}-q_{T}) \leq -\frac{\ell}{c} & \nu_{*}  \ \textrm{a.e.} \\[4pt]
U^{\nu_{*}}+\frac{1}{c}(-U^{\kappa_{M}}-q_{T}) \geq -\frac{\ell}{c} & \frac{1}{c}\kappa_{M}-\nu_{*} \ \textrm{a.e.}
\end{array}
\right.
\end{equation}

This is exactly the Euler-Lagrange formulation of the fact that $\nu_{*}$ is the minimiser of the functional
\[\I^{s}_{T,M_{*}}(\nu)= \iint_{[0,M]^2} -\log|u-v| \; d\nu(u) d\nu(v)+2 \int_{\mathbb R_+}\frac{-U^{\kappa_{M}}(u)-\frac{T}{2} u}{2\sqrt{M}-1}  \, d\nu(u)\]
on 
\[\L^{s}_{M_{*}}=\left\{\nu \in \Prob([0,M]) : \nu\leq \frac{\kappa_{|[0,M]}}{2\sqrt{M}-1}\right\}.\]

This minimisation problem may look rather complicated, but all we need to prove is that the support of its solution $\nu_{*}$ is an interval of the form $[a,M]$, that is,  an interval which contains $M$. The fact that the support of $\nu_{*}$ contains $M$ follows immediately from the fact that we chose $M$ outside the support of $\nu^{*}_{T}$.

In order to prove that the support of $\nu_{*}$ is an interval, we apply again Theorem \ref{supportconvexecontrainte}. For this, we need to prove that the external potential to which $\nu_{*}$ is submitted is convex. Since $q_{T}$ is linear, it is enough to prove that $U^{\kappa_{M}}$ is concave. This is done by direct computation: one establishes first that for all $x\in [0,M]$,
\[U^{\kappa_{M}}(x)=2\sqrt{x}\log \left(\frac{\sqrt{M}-\sqrt{x}}{\sqrt{M}+\sqrt{x}}\right)-2\sqrt{M} \log(M-x)+4\sqrt{M},\]
then that $\lim_{x\to 0} x^{\frac{3}{2}} (U^{\kappa_{M}})''(x) =0$, and finally that
\[\frac{d}{dx}\left(x^{\frac{3}{2}} (U^{\kappa_{M}})''(x)\right)=-\frac{\sqrt{Mx}}{(M-x)^{2}},\]
from which it follows that $(U^{\kappa_M})''$ is negative on $(0,M]$. This concludes the proof that $U^{\kappa_{M}}$ is concave.

{\em Step 5 : $a\neq b$.} 
We know by definition that  $a\leq b$. However, we still need to rule out the possibility that $a=b$. We do this by contradiction: let us assume that $a=b$. Then $\nu^{*}_{T}=\kappa_{|[0,a]}$, so that we must have $a=\frac{1}{4}$, and $\mu^{*}_{T}=\Leb_{|[-\frac{1}{2},\frac{1}{2}]}$. A direct computation then shows that for all $x\in \R$,
\[U^{\mu^{*}_{T}}(x)=1 + \left(x-\frac{1}{2}\right) \log \left|x-\frac{1}{2}\right|- \left(x+\frac{1}{2}\right)\log \left|x+\frac{1}{2}\right|.\]
In particular, $(U^{\mu^{*}_{T}})'(\frac{1}{2})=-\infty$ and no matter the value of $T$, the function $U^{\mu^{*}_{T}}+Q_{T}$ is strictly decreasing in a neighbourhood of $\frac{1}{2}$, which forbids $\mu^{*}_{T}$ from being the solution of the Euler-Lagrange version of the minimisation problem.
\end{proof}

Let us now turn to the determination of the values of $\alpha$ and $\beta$. We will use again a duality argument,  for which we found the inspiration in \cite{DraSaf00}.

\begin{proof}[Proof of Lemma \ref{MSab}]

For all reals $0\leq a \leq b$, let us define $J(a,b)=[-b,-a]\cup [a,b]$. Let us denote by $\eta^{*}$ the restriction of $\mu^{*}_{T}$ to $J=J(\alpha,\beta)$. Let us denote by $\pi_{\alpha}$ the restriction to $[-\alpha,\alpha]$ of the Lebesgue measure. By Lemma \ref{beckermann}  and Proposition \ref{levecontrainte}, the measure $\eta^{*}$ is equal to $1-2\alpha$ times the unique probability measure on $\Sigma_{\alpha}=\R\setminus (-\alpha,\alpha)$ which minimises the functional $\I_{Q_{T,\alpha}}$,
where we have set, for all $x\in \Sigma_{\alpha}$,
$$ Q_{T,\alpha}(x)= \frac{Q_{T}(x) + \Up{\pi_{\alpha}}(x)}{1-2\alpha}.$$
Our goal is to determine $\alpha$ and $\beta$ and this is equivalent to determining the support of $\eta^{*}$. To do this, we use the functional introduced by Mhaskar and Saff and which we discussed in Section \ref{sec:supportMS}.
According to Theorem \ref{theo:MS}, $J(\alpha,\beta)$ is, among all compact subsets of $\Sigma_{\alpha}$, one which minimises the functional $\MS_{Q_{T,\alpha}}$. From this, we deduce in particular that the function $b\mapsto \MS_{Q_{T,\alpha}}(J(\alpha,b))$ achieves its minimum on $(\alpha,+\infty)$ at $b=\beta$. An explicit computation will show that this function is differentiable and we find a first relation:
\begin{equation}\label{MS1}\tag{$\MS^{*}$}
\left.\frac{\partial}{\partial b}\MS_{Q_{T,a}}(J(a,b))\right|_{(a,b)=(\alpha,\beta)}=0.
\end{equation}

Since $\alpha$ is a boundary point of the closed set $\Sigma_{\alpha}$, we cannot let $a$ vary freely around $\alpha$ and it would be delicate to justify directly that $\frac{\partial}{\partial a}\MS_{Q_{T,c}}(J(a,b))$ vanishes at $(a,b,c)=(\alpha,\beta,\alpha)$. To go around this difficulty and to get a second relation between $\alpha$ and $\beta$, we use again a duality argument. Let us denote by $\pi_{\beta}$ the restriction to $[-\beta,\beta]$ of the Lebesgue measure, and define $\eta_{*}=\frac{1}{2\beta-1}(\pi_{\beta}-\mu^{*}_{T})=\frac{1}{2\beta-1}(\rm{Leb}_{|J(\alpha,\beta)}-\eta^{*})$. See Figure \ref{etaeta} for an illustration of these definitions.

\begin{figure}[h!]
\begin{center}
\scalebox{0.9}{\includegraphics{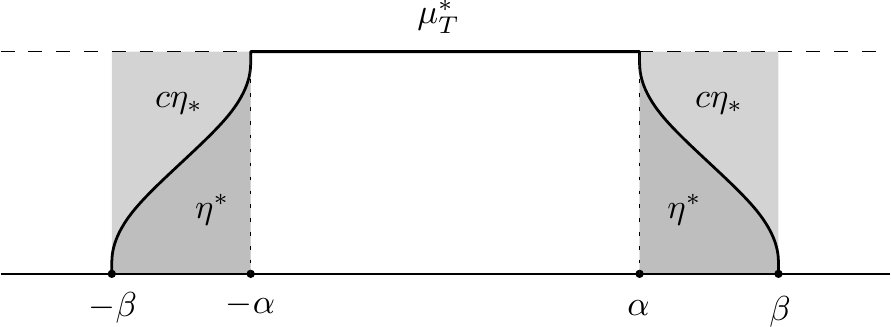} }
\caption{\label{etaeta} The measures $\eta^{*}$ and $c\eta_{*}$, with $c=2\beta-1$.}
\end{center}
\end{figure}

From the inequalities
\begin{equation}
\left\{\begin{array}{ll}
U^{\mu^{*}_{T}}+Q_{T} \leq \ell & \mu^{*}_{T}  \ \textrm{a.e.} \\
U^{\mu^{*}_{T}}+Q_{T} \geq \ell & \pi_{\beta}-\mu^{*}_{T} \ \textrm{a.e.}
\end{array}
\right.
\end{equation}
it follows that
\begin{equation}
\left\{\begin{array}{ll}
U^{\eta_{*}}+\frac{-U^{\pi_{\beta}}-Q_{T}}{2\beta-1}\leq -\frac{\ell}{2\beta-1} & \eta_{*}  \ \textrm{a.e.} \\[4pt]
U^{\eta_{*}}+\frac{-U^{\pi_{\beta}}-Q_{T}}{2\beta-1} \geq -\frac{\ell}{2\beta-1} & \frac{\pi_{\beta}}{2\beta-1}-\eta_{*} \ \textrm{a.e.}
\end{array}
\right.
\end{equation}
which can conveniently be rewritten as
\begin{equation}
\left\{\begin{array}{ll}
U^{\eta_{*}}+\frac{U^{\pi_{\beta}}+Q_{T}}{1-2\beta}\leq \frac{-\ell}{1-2\beta} & \eta_{*}  \ \textrm{a.e.} \\[4pt]
U^{\eta_{*}}+\frac{U^{\pi_{\beta}}+Q_{T}}{1-2\beta} \geq \frac{-\ell}{1-2\beta} & \frac{\pi_{\beta}}{2\beta-1}-\eta_{*} \ \textrm{a.e.}
\end{array}
\right.
\end{equation}

Since the support of $\frac{\pi_{\beta}}{2\beta-1}-\eta_{*}$ is the full interval $[-\beta,\beta]$, these inequalities express in Euler-Lagrange form the fact that $\eta_{*}$ is the unique minimiser of $\I_{Q_{T,\beta}}$ on $[-\beta,\beta]$. From this observation, it follows that the support of $\eta_{*}$, which is none other than $J(\alpha,\beta)$, achieves the minimum of the Mhaskar-Saff functional $\MS_{Q_{T,\beta}}$ among all compact subsets of $[-\beta,\beta]$. In particular, taking differentiability for granted,
\begin{equation}\label{MS2}\tag{$\MS_{*}$}
\left.\frac{\partial}{\partial a}\MS_{Q_{T,b}}(J(a,b))\right|_{(a,b)=(\alpha,\beta)}=0.
\end{equation}

Let us now compute the Mhaskar-Saff functional. In order to treat the two cases at once, let us choose $a,b,c$ positive reals such that $a\leq c \leq b$ and compute $\MS_{Q_{T,c}}(J(a,b))$. According to \eqref{eq:defMS}, we have
\[(1-2c)\MS_{Q_{T,c}}(J(a,b))=(1-2c)\int U^{\omega_{J(a,b)}} \; d\omega_{J(a,b)} + \int (Q+U^{{\rm{Leb}}_{|[-c,c]}}) \; d\omega_{J(a,b)}.\]

It is well known (see for example Section 14 of \cite{Widom}) that the equilibrium measure of the compact set $J(a,b)$ is given by
$$ d\omega_{J(a,b)} = \frac{|x|}{\pi \sqrt{(b^2-x^2)(x^2-a^2)}} \mathbf 1_{J(a,b)}(x)dx.$$ 
From there one computes directly 
\[\int x^2 d\omega_{J(a,b)} = \frac{a^2+b^2}{2}.\]
In order to compute the capacity of $J(a,b)$, one needs to compute the potential of the equilibrium measure $\omega_{J(a,b)}$. A first step for this is to compute its Stieltjes transform 
\[G^{\omega_{J(a,b)}}(z)=\frac{z}{\sqrt{(z^{2}-a^{2})(z^{2}-b^{2})}},\]
where as usual we take the branch of the square root which is a positive real number near real infinity. Taking the appropriate primitive and its real part, we find that for all $x \in [-a,a]$,
$$ \Up{\omega_{J(a,b)}}(x) = -\frac{1}{2} \log\left(\frac{a^2+b^2}{2}-x^2 + \sqrt{(b^2-x^2)(a^2-x^2)}\right)+\frac{1}{2} \log 2.$$
Taking $x=a$, we find
\[\int U^{\omega_{J(a,b)}} \; d\omega_{J(a,b)} =U^{\omega_{J(a,b)}}(a)=-\frac{1}{2}\log \frac{b^{2}-a^{2}}{4}.\]
Moreover, this is also the value of $U^{\omega_{J(a,b)}}$ on the whole set $J(a,b)$. Finally, we compute
\[\int \Up{{\rm{Leb}}_{|[-c,c]}} d\omega_{J(a,b)}= \int U^{\omega_{J(a,b)}} d{\rm{Leb}}_{|[-c,c]}=2\int_{0}^{c} U^{\omega_{J(a,b)}}(x)\; dx.\]
Since the potential of $\omega_{J(a,b)}$ is constant on $J(a,b)$, hence on $[a,c]$, we find
\[\int \Up{{\rm{Leb}}_{|[-c,c]}} d\omega_{J(a,b)}=2\int_{0}^{a} U^{\omega_{J(a,b)}}(x)\; dx +2(c-a)U^{\omega_{J(a,b)}}(a).\]
The last integral is relatively easily computed by parts, and we find, after some simplification,
\[\int \Up{{\rm{Leb}}_{|[-c,c]}} d\omega_{J(a,b)}=-c \log \frac{b^{2}-a^{2}}{4}- 2b\left(K\left(\frac{a}{b}\right)-E\left(\frac{a}{b}\right)\right).\]
Altogether, we find
\[(1-2c)\MS_{Q_{T,c}}(J(a,b))=-\frac{1}{2}\log \frac{b^{2}-a^{2}}{4}+\frac{T}{8}(a^{2}+b^{2})-2b\left(K\left(\frac{a}{b}\right)-E\left(\frac{a}{b}\right)\right),\]
which quite remarkably does not depend on $c$. From there, we see in particular that the Mhaskar-Saff functional is indeed differentiable. Therefore, Equations \eqref{MS1} and \eqref{MS2}
are indeed satisfied and if we denote by $m_{T}(a,b)$ the function on the right handside above, 
the numbers $\alpha$ and $\beta$ are then solutions of the equations
\[\frac{\partial}{\partial\alpha} m_{T}(\alpha,\beta)=0 \mbox{ and } \frac{\partial}{\partial\beta} m_{T}(\alpha,\beta)=0.\]

Taking into account the derivatives of the complete elliptic functions, and setting $k=\frac{\alpha}{\beta}$, these equalities entail
\[\frac{k}{1-k^{2}}+\frac{T}{4}k\beta^{2}-\frac{2k\beta}{1-k^{2}}E=-\frac{1}{1-k^{2}}+\frac{T}{4}\beta^{2} -2K\beta +\frac{2\beta}{1-k^{2}}E=0.\]
From these relations, one extracts
\[T\beta=4K \mbox{ and } \beta = \frac{1}{2E-(1-k^2)K}.\]
The proof is now complete. Indeed, the relation $\alpha=k\beta$ follows from our definition of $k$, the expression of $\beta$ above is exactly the expected one, and the relation between $T$ and $k$ follows immediately from the two equalities above. Moreover, elementary computations allow one to check that the map $k\mapsto 8EK-4(1-k^{2})K^{2}$ from $(0,1)$ to $\R$ is an increasing bijection from $(0,1)$ to $(\pi^{2},\infty)$. 
\end{proof}

\subsection{Derivation of the constrained weighted equilibrium measure} \label{sec:findweightedmin}
{\sl 

Our problem is now to find the density of the measure $\eta^{*}$, which is the restriction of the constrained minimiser $\mu^{*}_{T}$ on $J(\alpha,\beta)=[-\beta,-\alpha]\cup [\alpha,\beta]$. According to Proposition \ref{levecontrainte}, $\eta^{*}$ is $(1-2\alpha)$ times the probability measure on $J(\alpha,\beta)$ which minimises the functional $\I_{Q_{T,\alpha}}$, where for all $x\in \R$, we have set
\[Q_{T,\alpha}(x)=\frac{Q_{T}(x) +\Up{\pi_{\alpha}}(x)}{1-2\alpha}=\frac{1}{1-2\alpha}\left(\frac{T}{4}x^{2}+2\alpha+ (x-\alpha) \log |x-\alpha|- (x+\alpha)\log |x+\alpha|\right).\]
As in Section \ref{sec:findmin}, we do not care much about rigour in this subsection. A rigorous proof that the computations we are about to do indeed yield to the minimiser will be given in the next subsection.

Let us assume that $\eta^{*}$ admits a H\"older continuous density $\psi$ with respect to the Lebesgue measure. Let us also introduce the shorthand $J=J(\alpha,\beta)$. Recall that we denote by $\pi_{\alpha}$ the Lebesgue measure on $[-\alpha,\alpha]$. By differentiating formally the functional $\I_{Q_{T,\alpha}}$ at $\mu^{*}_{T}$ as we did in the case of Wigner matrices, or by writing the Euler-Lagrange formulation of the minimisation problem, we successively find that the function
\[x\mapsto \int_{J} -\log|x-y| \psi(y)\; dy +\frac{Tx^{2}}{4}+ \Up{\pi_\alpha} (x)\]
is constant on $J$, and that for all $x$ in the interior of $J$,
\[\PVint_{J} \frac{\psi(y)}{x-y}\; dy = \frac{Tx}{2}+(\Up{\pi_{\alpha}})'(x)=\frac{Tx}{2}+\log \frac{x-\alpha}{x+\alpha}.\]

In order to determine $\psi$, we introduce the Stieltjes transform of $\eta^{*}$:
\[\Gp{\eta^{*}}(z)=\int \frac{d\eta^{*}(y)}{z-y}=\int_{J} \frac{\psi(y)}{z-y} \; dy.\]
We are facing the same problem as in the case of Wigner matrices: we would like to know $\Gpp{\eta^{*}}-\Gpm{\eta^{*}}$ on the real axis but what we know is $\Gpp{\eta^{*}}+\Gpm{\eta^{*}}$. Let us introduce a new auxiliary function, which we still denote by $R$, and which is now analytic on $\C\setminus J$:
\[R(z)=\sqrt{(z^{2}-\alpha^{2})(z^{2}-\beta^{2})},\]
where we choose the branch of the square root which is a positive real number on $(\beta,\infty)$. For all $x\in J$, we now have
\[R_{\pm}(x)=\pm \sgn(x) i \sqrt{(\beta^{2}-x^{2})(x^{2}-\alpha^{2})}.\]
Thus, for all $x\in J$,
\[\left(\frac{\Gp{\eta^{*}}}{R}\right)_{\pm}(x)=-\frac{i\pi\psi(x)}{R_{+}(x)}\mp i \pi \frac{i}{\pi} \frac{\frac{T}{2}x+(\Up{\pi_{\alpha}})'(x)}{R_{+}(x)}\]
and, following the same reasoning as in the Wigner case, we expect the equality
\begin{equation}\label{Seta}
\Gp{\eta^{*}}(z)=\frac{i}{\pi}\sqrt{(z^{2}-\alpha^{2})(z^{2}-\beta^{2})}\int_{J} \frac{\frac{T}{2}y+(\Up{\pi_{\alpha}})'(y)}{(z-y)R_{+}(y)}\; dy
\end{equation}
to hold. If we manage to compute the right-hand side of \eqref{Seta}, it will be easy to compute $\psi$ using the Sokhotski formula \eqref{PS}. To start with, we have for all $x\in J$
\[(\Up{\pi_{\alpha}})'(x)=\log\frac{x-\alpha}{x+\alpha}.\]
Let us introduce the function
\[L(z)=\frac{T}{2}z+\log\frac{z-\alpha}{z+\alpha},\]
which is holomorphic on $\C\setminus [-\alpha,\alpha]$, and satisfies $L(x)=\frac{T}{2}x+(\Up{\pi_{\alpha}})'(x)$ for all $x\in J$.

We are going to compute a contour integral in two different ways. Let $z$ be a point in $\C\setminus [-\beta,\beta]$. Let $\gamma$ be a contour which surrounds once both $z$ and the interval $[-\beta,\beta]$ (see Figure \ref{contour}). We want to compute the integral 
\[\int_{\gamma} \frac{L(w)}{R(w)}\frac{dw}{z-w}.\]
On one hand, we have $\frac{L(w)}{R(w)}\frac{1}{z-w}=O(\frac{1}{w^{2}})$ near infinity, so that, by deforming $\gamma$ into a large circle, we see that the integral is zero. 

\begin{figure}[h!]
\begin{center}
\includegraphics[width=8cm]{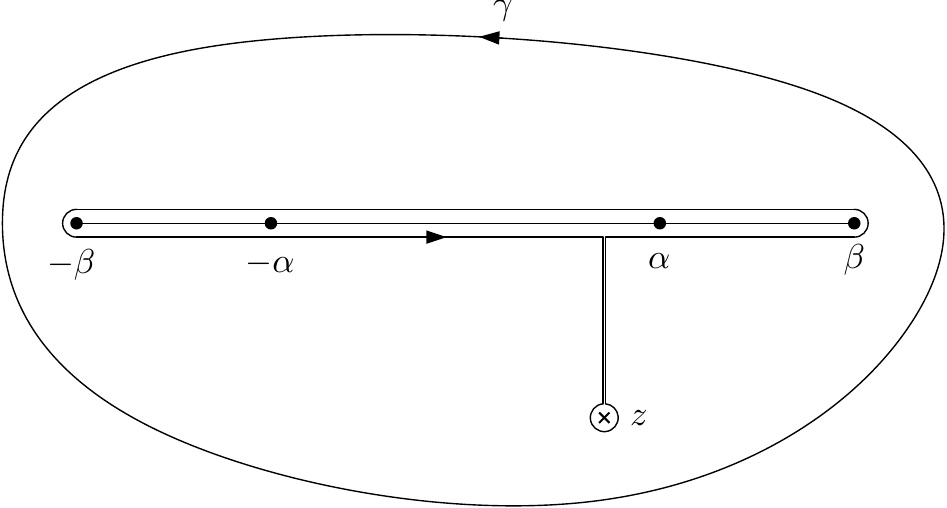}
\caption{The contour $\gamma$ and its deformation. \label{contour}}
\end{center}
\end{figure}
On the other hand, the same contour can be deformed so as to make a small circle around $z$ and to go twice along the interval $[-\beta,\beta]$, once slightly below and once slightly above. We find
\[0=-2i\pi \frac{L(z)}{R(z)}+\int_{-\beta}^{\beta} \left(- \frac{L_{+}(y)}{R_{+}(y)} +\frac{L_{-}(y)}{R_{-}(y)}\right) \frac{dy}{z-y}.\]
By carefully keeping track of the values of $L_{\pm}$ and $R_{\pm}$ on the intervals $[-\beta,-\alpha]$, $[-\alpha,\alpha]$ and $[\alpha,\beta]$, one finds
\[\int_{J} \frac{L(y)}{R_{+}(y)} \frac{dy}{z-y}=-i\pi\frac{L(z)}{R(z)} -i\pi \int_{-\alpha}^{\alpha} \frac{1}{R(y)} \frac{dy}{z-y},\]
from which we conclude that
\begin{align*}
\Gp{\eta^{*}}(z)&=L(z)+R(z)\int_{-\alpha}^{\alpha} \frac{1}{R(y)} \frac{dy}{z-y}\\
&=\log \frac{z-\alpha}{z+\alpha} +\frac{T}{2}z-\sqrt{(z^{2}-\alpha^{2})(z^{2}-\beta^{2})} \int_{-\alpha}^{\alpha}  \frac{dy}{(z-y)\sqrt{(\beta^{2}-y^{2})(\alpha^{2}-y^{2})}}.
\end{align*}
Adding the Stieltjes transform of $\pi_{\alpha}$ on one hand, setting $k=\frac{\alpha}{\beta}$ and performing elementary manipulations to the integral, we obtain our candidate for the Stieltjes transform of the full constrained weighted equilibrium measure $\mu^{*}_{T}$:
\begin{equation}\label{Smini}
H_T(z)=\frac{T}{2}z-\frac{2}{\beta z}\sqrt{(z^{2}-\alpha^{2})(z^{2}-\beta^{2})} \int_{0}^{1} \frac{ds}{(1-\frac{\alpha^{2}}{z^{2}}s^{2})\sqrt{(1-s^{2})(1-k^{2}s^{2})}}.
\end{equation}
This last integral is an elliptic integral of the third kind, for which one standard notation is
\begin{equation}\label{Pi}
\Pi(\nu;k)=\int_{0}^{1} \frac{ds}{(1-\nu s^{2})\sqrt{(1-s^{2})(1-k^{2}s^{2})}}.
\end{equation}
For each $k\in (0,1)$, this is a function of $\nu$ analytic in $\C\setminus [1,\infty)$, and we can deduce from \eqref{Smini} our candidate for the value of the density of $\eta^{*}$, that is of $\mu^{*}_{T}$, on $J$:
\[\forall x\in J,\ \psi(x)=-\frac{1}{2i\pi}((H_{T})_+-(H_T)_-)(x)=\frac{2}{\pi\beta |x|} \sqrt{(\beta^{2}-x^{2})(x^{2}-\alpha^{2})}\; \Pi\left(\frac{\alpha^{2}}{x^{2}};k\right).\]
Putting everything together, we obtain the following candidate, which we denote by $\phi_{T}$, for the density of $\mu^{*}_{T}$:
\begin{equation}\label{densitesat}
\phi_T(x)=\1_{[-\alpha,\alpha]}(x)+\frac{2}{\pi\beta |x|} \sqrt{(\beta^{2}-x^{2})(x^{2}-\alpha^{2})}\; \Pi\left(\frac{\alpha^{2}}{x^{2}};k\right) \1_{[-\beta,-\alpha]\cup [\alpha,\beta]}(x).
\end{equation}

It is interesting to note that we did not use the actual values of $\alpha$ and $\beta$ in our computation. At this point in the case of Wigner matrices, we determined the width of the support of the minimiser by considering the asymptotic behaviour of its Stieltjes transform. The same could be done here and, if we had not already rigorously determined the values of $\alpha$ and $\beta$, we could find them by writing that $H_T(z)=z^{-1}+O(z^{-2})$ near infinity (see also the proof of Proposition \ref{calculpotmini} below).

The result of our derivation of $\mu^{*}_{T}$ is illustrated in Figure \ref{figdensite} below.

\begin{figure}[h!]
\begin{center}
\includegraphics[width=4.5cm]{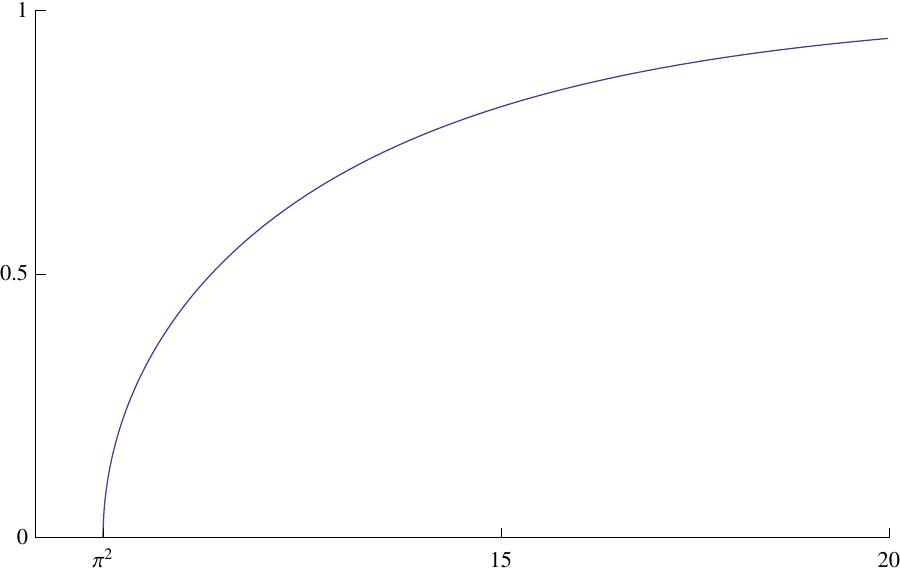}\hspace{0.3cm}\includegraphics[width=4.5cm]{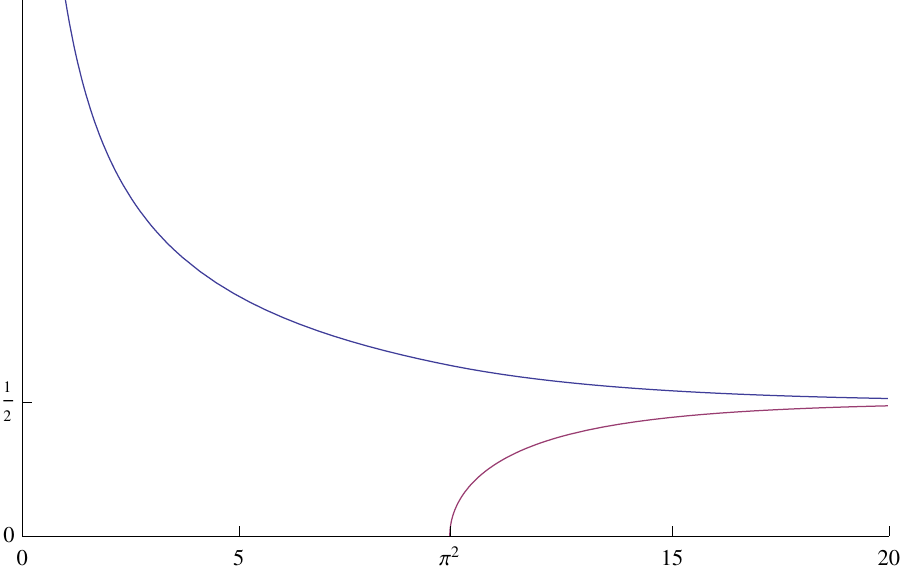}\hspace{0.5cm}\raisebox{-5mm}{\includegraphics[width=4.5cm]{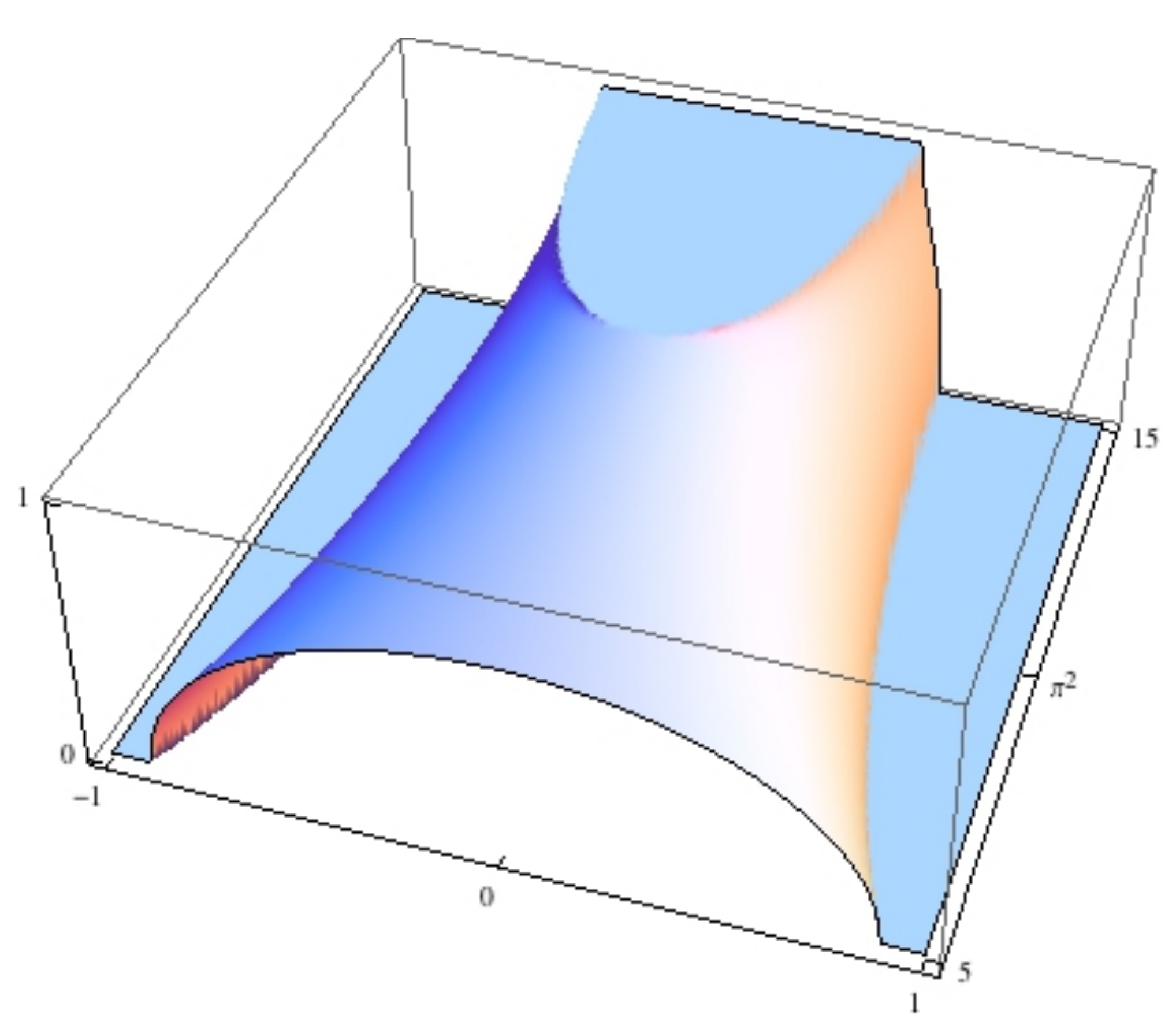}}
\caption{\label{figdensite} From left to right: $k$ as a function of $T$; $\alpha$ (red) and $\beta$ (blue) as functions of $T$; the density of the minimising measure $\mu^{*}_{T}$ as a function of $T$.}
\end{center}
\end{figure}
}

\subsection{Verification of the expression of the equilibrium measure} We are now equipped with an explicit candidate for the minimiser $\mu^{*}_{T}$, given by \eqref{densitesat} and Lemma \ref{MSab}. All we need to prove is that the restriction to $J$ of the measure $\phi_{T}(x)dx$ is, as we expect, $(1-2\alpha)$ times the minimiser of $\I_{Q_{T,\alpha}}$ on $\Sigma_{\alpha}$. For this, according to Theorem \ref{potquesurlesupport}, it suffices to check that the potential of $\phi_{T}(x)dx$ is constant on $J$. 

The next proposition concludes the whole process of determination of the minimiser.

\begin{proposition} For every $T>\pi^{2}$, the unique minimiser of $\I_{Q_{T}}$ on $\L(\R)$ is the measure $\mu^{*}_{T},$ whose density $\phi_T$ is given by \eqref{densitesat}. Moreover, $G^{\mu^{*}_{T}}$, the Stieltjes transform of $\mu^{*}_{T}$, is the function $H_{T}$ given by \eqref{Smini}.

\end{proposition}

\begin{proof}  From its explicit expression \eqref{Smini}, we can expand $H_T$ near infinity. We find 
\begin{align*}
H_T(z)&=\frac{T}{2}z-\frac{2z}{\beta}\left(1-\frac{\alpha^{2}+\beta^{2}}{2z^{2}}\right)\left(K+\frac{\alpha^{2}}{z^{2}}\int_{0}^{1} \frac{\frac{1}{k^{2}}(1-(1-k^{2}s^{2}))}{\sqrt{(1-s^{2})(1-k^{2}s^{2})}}\; ds\right) + O\left(\frac{1}{z^{2}}\right)\\
&=\left(\frac{T}{2}-\frac{2K}{\beta}\right) z + \frac{1}{z}\left(\beta(2E-(1-k^{2})K)\right)+ O\left(\frac{1}{z^{2}}\right)
\end{align*}
and the relations between $T$, $k$ and $\beta$ imply precisely that $H_T(z)=z^{-1}+O(z^{-2})$.

Now, the function $\phi_T$ defined in \eqref{densitesat} is positive and smooth on $\R\setminus\{\pm \alpha,\pm\beta\}$ and one can check directly that its lack of smoothness at these four points is that of  square root at the origin. It is thus $\frac{1}{2}$-H\"older continuous, and the Sokhotski formula \eqref{PS} is true  at every point of $\R$ for its Stieltjes transform that we will denote by $G^{\phi_T}$. Hence, 
$G^{\phi_T}$ and $H_T$ have exactly the same jumps across the real axis. The difference between these two functions, which is analytic on $\C\setminus \R$ and continuous on $\C$, is thus analytic on $\C$, for example as a consequence of Morera's theorem. By the computation above, we know that $H_T$ is equivalent to  $z^{-1}$ near infinity, and $G^{\phi_T}$ is equivalent to $z^{-1}\int \phi_T(x)dx.$
Their difference is therefore entire and bounded, hence constant, and in fact equal to zero. This proves that $H_T$ is the Stieltjes transform of $\phi_T$, and also that $\phi_{T}$ is the density of a probability measure.

In order to see that $\phi_{T}$ is the density of the minimiser, let us denote by $G^\psi$ the Stieltjes transform of $\psi$, which is the restriction of $\phi_T$ to $J$. The computations made in the previous subsection give directly that the function $G^\psi$ satisfies
 $$ \forall x \in J,\quad  (G^\psi)_+(x) + (G^\psi)_-(x) = 2\left(\frac{Tx}{2} + \log \frac{x-\alpha}{x+\alpha} \right).$$
By regularity of the function $\psi$ and again by \eqref{PS}, the left-hand side is equal to $2\, \PVint_J \frac{\psi(y)}{x-y}dy$. Therefore, the function $x \mapsto \int_J -\log|x-y|\psi(y)dy + \frac{Tx^2}{4} + U^{\pi_\alpha}(x)$ is constant on $J$. An application of Proposition \ref{levecontrainte} concludes the proof.
\end{proof}


Figure \ref{potsat} illustrates the shape of the potential of the measure $\mu^{*}_{T}$.

\begin{figure}[h!]
\begin{center}
\includegraphics[width=6cm]{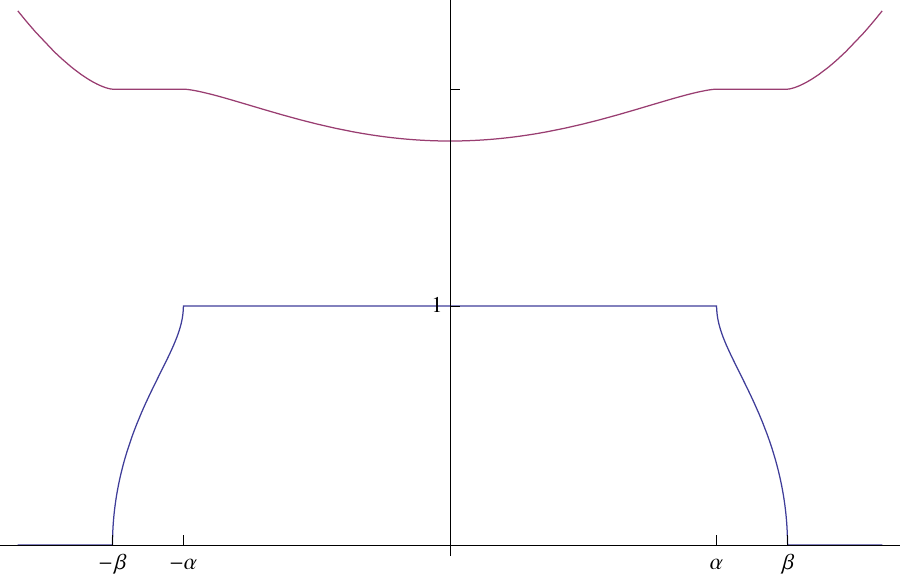}
\caption{\label{potsat} The density of the minimising measure $\mu^{*}_{T}$ and the graph of the function $x\mapsto \Up{\mu^{*}_{T}}(x)+\frac{T}{4}x^{2}$, for $T$ slightly larger than $14$.}
\end{center}
\end{figure}

\section{The Douglas-Kazakov phase transition}\label{sec:transition}

\subsection{Computation of the free energy}

We now have all the tools in hand to prove our main result, Theorem \ref{PT3}. Recall from Proposition \ref{limZ} that the free energy $F$, which is a function of $T$, and the regularity of which is our main concern, is the difference between an affine function of $T$ and the function $M$ defined by
\[M(T)=\I_{T}(\mu^{*}_{T}).\]
Our first task is to compute $M$. Recall that for $T> \pi^{2}$, the numbers $k\in (0,1)$, $\beta>0$ and $\alpha\in (0,\beta)$ are determined as functions of $T$ by Lemma \ref{MSab}.

On the way of finding an explicit expression for the function $M,$ that will eventually appear in Proposition \ref{calculeM}, we will need to introduce the incomplete elliptic integrals of the first and second kind, which for each $k\in (0,1)$ are the following analytic functions in $\C\setminus \big((-\infty,-1]\cup [1,\infty])\big)$:
\[F(z;k)=\int_{0}^{z} \frac{1}{\sqrt{(1-s^{2})(1-k^{2}s^{2})}}\; ds\ \mbox{ and } \ E(z;k)=\int_{0}^{z} \sqrt{\frac{1-k^{2}s^{2}}{1-s^{2}}}\; ds,\]
and to find an expression of the logarithmic potential of $\mu_T^*$ at the points $\alpha$ and $\beta$. This is what the next lemma does. 

\begin{proposition}\label{calculpotmini}
For all $z\in \C\setminus (-\infty,\beta]$, we have
\begin{align}
\int\log(z-y) \; d\mu^{*}_{T}(y)&=-\log\frac{\sqrt{z^{2}-\alpha^{2}}+\sqrt{z^{2}-\beta^{2}}}{2}-2E z F\left(\frac{\beta}{z};k\right)+2KzE\left(\frac{\beta}{z};k\right)\nonumber \\
&\hspace{4cm} -\frac{K}{2\beta}\left((2z)^{2}-\left(\sqrt{z^{2}-\alpha^{2}}+\sqrt{z^{2}-\beta^{2}}\right)^{2}\right)+1,\label{gmin}
\end{align}
where $\log$ is the branch of the logarithm on $\C\setminus \R_{-}$ such that to $\log 1=0$.
\end{proposition}


\begin{proof} It suffices to check that the right-hand side of \eqref{gmin} is equal to 
$\log z+O(z^{-1})$ near infinity and that its derivative is equal to $G^{\mu^{*}_{T}}$.
\end{proof}

We now go to the main object of this subsection, namely the computation of $M.$

\begin{proposition}\label{calculeM}
The function $M$ is given by 
$$  M(T)=\left\{\begin{array}{ll}
                   \frac{1}{2} \log T + \frac{3}{4} & \textrm{ if } T \le \pi^2, \\[3pt]
                   \frac{3}{2} -\frac{1}{2} \log\left(\frac{\beta^2}{4}(1-k^2)\right) -\frac{2}{3} K\beta (1+k^2) -\frac{1}{12} K^2\beta^2(1-k^2)^2& \textrm{ if } T > \pi^2.
                  \end{array}
\right.$$
\end{proposition}

The proof of this proposition is a direct computation. It can be seen as elementary, but we find it difficult enough to deserve a fairly detailed presentation. Moreover, we do not believe that this full expression has appeared in the literature before.
  
\begin{proof}
Let us use the notation $Q(z) = Q_T(z) = \frac{T}{4}z^2$, for all $z \in \C$. To simplify the notations, let us denote $\mu^{*}_{T}$, $\Up{\mu_T^*}$, $\Gp{\mu_T^*}$ respectively by $\mu^{*}$, $U$, $G$. Let us also define
\[g(z)=\int\log(z-y) \; d\mu^{*}(y).\]
With this notation,
\[M(T) = \int (U+2Q) \; d\mu^*\] and we recall that $U=-\Re g$ whereas $G^\prime=g.$

Let us start with the expression of $M$ in the subcritical regime $T\leq \pi^{2}$. In this case, we write
\[ M(T) = \int (U+2Q)\; d\mu^*= \int (U+Q) \; d\mu^* +  \int Q \; d\mu^*.\]
We know that the function $U+Q$ is constant on the support of $\mu^*=\sigma_{1/T}$, which is the interval $[-\frac{2}{\sqrt{T}},\frac{2}{\sqrt{T}}]$. From \eqref{lepotWig}, we know that this constant is $\frac{1}{2}( \log T + 1)$. For the second term, a simple change of variable yields  $\int Q \; d\mu^*=\frac{1}{4},$ from where we get that $M(T) =  \frac{1}{2} \log T + \frac 3 4,$ whenever $ T \le \pi^2.$

Let us now turn to the supercritical regime $T>\pi^{2}$. We will use the notation $J =J(\alpha,\beta)= [-\beta, -\alpha]\cup[\alpha, \beta]$. Let us split $M(T)$ as follows:
$$ M(T) = \int (U+2Q) \; d\mu^* = \underbrace{\int_J (U+Q)\;  d\mu^*}_{\raisebox{-3pt}{\textcircled{\scriptsize 1}}} + \underbrace{\int_{-\alpha}^\alpha Q(t) \; dt}_{\raisebox{-3pt}{\textcircled{\scriptsize 2}}}+ \underbrace{\int Q \; d\mu^*}_{\raisebox{-3pt}{\textcircled{\scriptsize 3}}} +\underbrace{\int_{-\alpha}^\alpha U(t) \; dt}_{\raisebox{-3pt}{\textcircled{\scriptsize 4}}}. $$

We will compute each term separately.\\

{\textbf 1.} The function $U+Q$ is constant almost everywhere on the set $J$ and continuous, so that it is constant everywhere on $J$. It is thus enough to evaluate it at $\beta$. In order to evaluate $U$ at $\beta$, we use \eqref{gmin} and the relation
\[U(\beta)=\lim_{x\downarrow \beta} -\Re g(x)=-\log \frac{\sqrt{\beta^{2}-\alpha^{2}}}{2}-\frac{K}{2\beta}(\alpha^{2}+\beta^{2})+1-K\beta.\]
Thus, 
\begin{equation}
 \label{QB}
 (U+Q)(\beta) = -\frac{1}{2} \log \frac{\beta^2-\alpha^2}{4}- \frac{K\beta}{2}(1+k^2)+1,
\end{equation}
so that 
\begin{equation}
 \label{terme1}
 \int_J (U+Q) \; d\mu^*  = (1-2\alpha)(U+Q)(\beta) = (2k\beta -1)\left(\frac{1}{2} \log \frac{\beta^2-\alpha^2}{4}+ \frac{K\beta}{2}(1+k^2)-1\right).
\end{equation}

{\textbf 2.}  The second term is the easiest: since $Q(z)=\frac{K}{\beta}z^{2}$ and $\alpha=k\beta$, we have
  \begin{equation}
 \label{terme2}
 \int_{-\alpha}^\alpha Q(t) \; dt = \frac{2}{3}K\beta^2k^3.
 \end{equation}
 
{\textbf 3.} To compute the third term, we use the fact that $\int y^2\; d\mu^*(y)$ is the coefficient of $\frac{1}{z^3}$ in the expansion of $G(z)=\int \frac{1}{z-y} \; d\mu^{*}(y)$ near infinity.
 
Recall  from \eqref{Pi} the definition of the elliptic integral of the third kind. By Formula 17.7.6 in \cite{AbraSte64}, we get that for any $z\in [\beta, \infty),$ 
 $$ \Pi\left(\frac{\alpha^2}{z^2};k\right)= K + \frac{\beta z}{\sqrt{(z^2-\beta^2)(z^2-\alpha^2)}} \left(K E\left(\frac \beta z ; k\right)- E F\left(\frac \beta z ; k\right) \right).$$
We can thus rewrite \eqref{Smini} and find, for all $z \in [\beta,\infty)$,
\begin{equation}\label{GT1}
 G(z) = 2E F\left(\frac{\beta}{z};k\right) -2K E\left(\frac{\beta}{z};k\right) - \frac{2K}{\beta}z\left(\sqrt{\left(1-\frac{\alpha^2}{z^2}\right)\left(1-\frac{\beta^2}{z^2}\right)}-1\right).
\end{equation}

One has the following expansions in the vicinity of zero:
 $$ F(u;k) = u + \frac{1+k^2}{6} u^3 + \mathcal O(u^5) \quad \textrm{ and } \quad E(u;k) = u + \frac{1-k^2}{6} u^3 + \mathcal O(u^5).$$
 Thus, 
 $$ G(z) = \frac{1}{z}+ \left(E\frac{\beta^3}{3}(1+k^2) + K\beta^3 \left(\frac{(1+k^2)^2}{4}- \frac{1-k^2}{3}-k^2\right)\right)\frac{1}{z^3}+ o\left(\frac{1}{z^3}\right).$$
so that
  \begin{equation}
 \label{terme3}
 \int Q \; d\mu^*  =  \frac{K\beta}{6} (1+k^2)+ \frac{K^2\beta^2 }{12}(1-k^2)^2.
 \end{equation}

{\textbf 4.} The last term is the most delicate to compute. 

By using the symmetries of the functions $U$ and $g$, and performing integrations by parts, we have
 \begin{eqnarray*}
  \int_{-\alpha}^\alpha U(t) dt & = & - \Re \int_{-\alpha}^\alpha g(t) dt = -2 \Re \int_{0}^\alpha g(t) dt =- 2 \Re \int_{0}^\alpha g_+(t) dt \\
  & =&  2\alpha U(\alpha) + 2 \Re \int_{0}^\alpha t G_+(t) dt =  2\alpha U(\alpha)+ \alpha^2 \Re G_+(\alpha) - \Re  \int_{0}^\alpha t^2 G_+^\prime(t)dt.
 \end{eqnarray*}
We compute $U(\alpha)$ using again the fact that $U+Q$ is constant on $J$:
$$ U(\alpha) = (U+Q)(\beta) - Q(\alpha) = -\frac{1}{2} \log \frac{\beta^2-\alpha^2}{4}- \frac{K\beta}{2}(1+3k^2)+1.$$
We then compute $G_+$ from the \eqref{GT1}: for $x \in [-\alpha, \alpha]$ we find
\begin{equation}\label{Gplus}
G_+(x) = 2E F_-\left(\frac{\beta}{x}; k\right) -2 K E_-\left(\frac{\beta}{x}; k\right) + \frac{2K}{\beta} \left(\frac{\sqrt{(\alpha^2-x^2)(\beta^2-x^2)}}{x}+x\right).
\end{equation}

One can check that, for $x \in [-\alpha, \alpha],$
\begin{eqnarray}
 F_-\left(\frac{\beta}{x}; k\right) & = & K -\int_{1/k}^{\beta/x} \frac{ds}{\sqrt{(s^2-1)(k^2s^2-1)}} + i \int_1^{1/k}  \frac{ds}{\sqrt{(s^2-1)(1-k^2s^2)}}\\
 E_-\left(\frac{\beta}{x}; k\right) & = & E +\int_{1/k}^{\beta/x} \sqrt{\frac{k^2s^2-1}{s^2-1}}ds + i \int_1^{1/k}  \sqrt{\frac{1-k^2s^2}{s^2-1}}ds, 
\end{eqnarray}
so that $\Re G_+(\alpha) = 2Kk.$

Differentiating \eqref{Gplus}, we get that, for all $x \in [-\alpha, \alpha],$
$$ x^2 \Re G_+^\prime(x) = \frac{2K}{\beta}x^2 + \frac{2K}{\beta} \frac{x^4}{\sqrt{(\alpha^2-x^2)(\beta^2-x^2)}} + 2(E-K) \beta \frac{x^2}{\sqrt{(\alpha^2-x^2)(\beta^2-x^2)}}.$$

One can check that 
$$ \int_0^z  \frac{x^2}{\sqrt{(\alpha^2-x^2)(\beta^2-x^2)}}dx = \beta \left(F\left(\frac{z}{\alpha}; k\right)- E\left(\frac{z}{\alpha}; k\right)\right) $$
and, computing the derivative of $x \mapsto x \sqrt{(\alpha^2-x^2)(\beta^2-x^2)}$, that
\begin{align*} \int_0^z  \frac{x^4}{\sqrt{(\alpha^2-x^2)(\beta^2-x^2)}}dx =&\\
&\hspace{-2cm}
\frac{z}{3} \sqrt{(\alpha^2-z^2)(\beta^2-z^2)} + \frac{\beta}{3}(\alpha^2+2\beta^2)F\left(\frac{z}{\alpha}; k\right)- 
\frac{2\beta}{3}(\alpha^2+\beta^2)E\left(\frac{z}{\alpha}; k\right).
\end{align*}

Therefore, using again the relation $2E\beta =1+(1-k^2)K\beta,$ we get 
$$\Re  \int_{0}^\alpha t^2 G_+^\prime(t)dt= \frac{2}{3}K\beta^2 k^3 + \frac{K\beta}{3}(1+k^2)+ \frac{K^2 \beta^2}{6}(1-k^2)^2-\frac{1}{2}.$$

There remains to put all the pieces together, and we get the expected expression for $M(T)$.
\end{proof}

\subsection{The jump of the third derivative} We finally turn to the proof of Theorem \ref{PT3}.

\begin{proof}[Proof of Theorem \ref{PT3}] It follows immediately from Proposition \ref{calculeM} that the function $M$, hence the free energy, is a smooth function of $T$ on $\R_{+}\setminus\{\pi^{2}\}$. We want to check that the function $M^{(3)}$, which is the third derivative of $M$, has a discontinuity  of the first kind at $T = \pi^2.$

It follows from Lemma \ref{MSab} and from the definitions \eqref{defEK} of the complete elliptic integrals that 
$$ \lim_{T \downarrow \pi^2} k(T) = 0, \quad \lim_{T \downarrow \pi^2} K(k(T)) = K(0) = \frac{\pi}{2} \textrm{ and } \lim_{T \downarrow \pi^2} \beta(k(T)) = \frac{2}{\pi}.$$
Therefore,
$$ \lim_{T \downarrow \pi^2} M(T) =   \frac{1}{2} \log(\pi^2)+\frac{3}{4}=M(\pi^{2}),$$
so that $M$ is continuous at $T = \pi^2.$

We will now compute the derivatives of $M$ at $\pi^{2}$ from both sides. It turns out to be convenient to introduce the variable $m = 1-k^2$. Considering $E$ and $K$ as functions of $m,$ one easily gets that, for $m \in (0,1),$
$$ \frac{\partial E}{\partial m} = \frac{1}{2(m-1)} (E-K) \textrm{ and }  \frac{\partial K}{\partial m} = \frac{E}{2m(m-1)} - \frac{K}{2(m-1)}.$$

The relation $2E\beta =1+(1-k^2)K\beta$ can be rewritten $\displaystyle \beta = \frac{1}{2E-mK},$
so that 
$$  \frac{\partial \beta}{\partial m} = -\frac{\beta}{4(m-1)}(1-mK\beta) \textrm{ and }
 \frac{\partial (K\beta)}{\partial m} = \frac{(1-mK\beta)^2}{4m(m-1)}.$$
 
 The function $M$ itself can be rewritten as
 $$ M= -\log \beta -\frac{1}{2} \log m - \frac{2}{3} K\beta(2-m) -\frac{1}{12} m^2K^2\beta^2,$$
and we find that the derivative $\frac{\partial M}{\partial m}$  factorises very nicely under the form
 $$ \frac{\partial M}{\partial m} = \frac{(1-mK\beta)(1+mK\beta)}{24m(m-1)}(m^2K\beta-2m+4).$$
 Moreover, from the relation $\beta T= 4K,$ one gets that
 $$  \frac{\partial m}{\partial T} = \frac{m(m-1)\beta^2}{(1-mK\beta)(1+mK\beta)},$$
 so that we finally get, that for any $T>\pi^2,$
 $$ M^\prime(T) =\frac{d M}{d T}= \frac{\beta^2}{24}(m^2K\beta-2m+4).$$
 This formula is exactly the one predicted by Douglas and Kazakov when they introduced the model (see Equation (35) in \cite{DouKaz93}).

 From this expression of $M'(T)$, one can check that 
 $$  \lim_{T \downarrow \pi^2} M^\prime(T) = \frac{1}{2\pi^2},$$
 so that $M$ is continuously derivable at $\pi^2$. We keep the notation $M'$ for the first derivative of $M$ with respect to $T$.
 
 From the same formul{\ae}  as above, one gets that 
 $$  \frac{\partial M^\prime}{\partial m}= -\frac{m\beta^2}{32(m-1)}(1-mK\beta)(1+mK\beta), $$
 so that $$ M^{\prime\prime}(T) =\frac{d^{2} M}{d T^{2}}=  -\frac{m^2\beta^4}{32}.$$
 This formula yields
 $$  \lim_{T \downarrow \pi^2} M^{\prime\prime}(T) = -\frac{1}{2\pi^4},$$
 so that $M$ is twice continuously derivable at $\pi^2.$\\
 
Then
$$ M^{(3)}(T) =\frac{d^{3}M}{dT^{3}}= \frac{m^2 \beta^6}{32} \left(\frac{m}{1+mK\beta} -\frac{2(m-1)}{(1-mK\beta)(1+mK\beta)}\right).$$
By L'H\^opital's rule, we find 
$$  \lim_{T \downarrow \pi^2} M^{(3)}(T) = \frac{3}{\pi^6},  \ \textrm{ whereas }\quad  \lim_{T \uparrow \pi^2} M^{(3)}(T) = \frac{1}{\pi^6},$$
so that $ M^{(3)}$ has a discontinuity  of the first kind at $T = \pi^2.$

Since the free energy is an affine function of $T$ minus the function $M$, we find the announced limits on the left and on the right for the third derivatives of $F$.
\end{proof}

\bibliographystyle{acm}
\bibliography{BibEsaim}

\end{document}